\documentclass[a4paper,12pt]{article}

\usepackage {amsfonts, amssymb, amscd,amsthm, amsmath, eucal}

\usepackage{amsbsy}
\usepackage[all]{xy}

\usepackage{color}

\theoremstyle{plain} {
  \newtheorem{thm}{Theorem}[section]
  \newtheorem{defn}[thm]{Definition}
  \newtheorem{cor}[thm]{Corollary}
  \newtheorem{lem}[thm]{Lemma}
  \newtheorem{prop}[thm]{Proposition}
  \theoremstyle{definition}
  \newtheorem{rem}[thm]{Remark}
    
  \theoremstyle{plain}
  
  \newtheorem{notation}[thm]{Notation}
  
  \newtheorem{sssection}[thm]{}
}

{

}
\renewcommand{\subsubsection}{\sssection\rm}

\newcommand{\bE}{\mathbf E}
\newcommand{\bG}{\mathbf G}
\newcommand{\bH}{\mathbf H}

\newcommand{\bP}{\mathbf P}
\newcommand{\bT}{\mathbf T}

\newcommand{\bu}{\mathbf u}

\newcommand{\bC}{\mathbf C}
\newcommand{\bw}{\mathbf w}
\newcommand{\bs}{\mathbf s}

\newcommand{\bS}{\mathbf S}

\newcommand{\cE}{\mathcal E}
\newcommand{\cF}{\mathcal F}
\newcommand{\cG}{\mathcal G}
\newcommand{\cH}{\mathcal H}

\newcommand{\cO}{\mathcal O}
\newcommand{\cP}{\mathcal P}

\newcommand{\cT}{\mathcal T}
\newcommand{\cX}{\mathcal X}

\renewcommand{\P}{\mathbb P}

\DeclareMathOperator{\spec}{Spec}

\newcommand{\pr}{\text{\rm pr}}

\newcommand{\Aff}{\mathbf {A}}
\newcommand{\Pro}{\mathbf {P}}

\newcommand \xra {\xrightarrow }

\newcommand \hra {\hookrightarrow }

\newcommand{\E}{\mathcal E}

\renewcommand{\P}{\mathbb P}

\renewcommand \phi\varphi

\begin{document}

\title{On the Gille theorem for the relative projective line
}

\author{Ivan Panin\footnote{
St. Petersburg Department of Steklov Mathematical Institute, nab. r. Fontanki 27, 191023, St. Petersburg, Russia;
paniniv@gmail.com}
, \ Anastasia Stavrova\footnote{
St. Petersburg Department of Steklov Mathematical Institute, nab. r. Fontanki 27, 191023, St. Petersburg, Russia;
anastasia.stavrova@gmail.com}}

\maketitle

\begin{abstract}
Let $X$ be a Noetherian separated scheme. Let $G$ be a reductive $X$-group scheme,
and let $E$ be a principal $G$-bundle over $\mathbb{P}^1_X$. We prove that
if the restriction of $E$ to $\infty\times X$ is Zariski locally trivial, then $E$ is itself Zariski locally trivial.
\end{abstract}

\section{Main results}\label{Introduction}
All the rings in the present text are commutative and Noetherian.
All the schemes in the present text are Noetherian and separated.\\\\
Let $R$ be a ring. Recall that an $R$-group scheme $\bG$ is called \emph{reductive},
if it is affine and smooth as an $R$-scheme and if, moreover,
for each algebraically closed field $\Omega$ and for each ring homomorphism $R\to\Omega$ the scalar extension $\bG_\Omega$ is
a connected reductive algebraic group over $\Omega$. This definition of a reductive $R$-group scheme
coincides with~\cite[Exp.~XIX, Definition~2.7]{SGA3}.

Assume that $X$ is a  scheme, $\bG$ is a reductive $X$-group scheme.
Recall that a $X$-scheme $\cG$ with an action of $\bG$ is called \emph{a principal $\bG$-bundle over $X$},
if $\cG$ is faithfully flat and quasi-compact over $X$ and the action is simply transitive,
that is, the natural morphism $\bG\times_X\cG\to\cG\times_X\cG$ is an isomorphism, see~\cite[Section~6]{Gr4}.
It is well known that such a bundle is trivial locally in \'etale topology but in general not in the Zariski topology.
The first main result of our paper is the following theorem.

\begin{thm}\label{th:X}
Let $X$ be a  scheme. Let $\bG$ be a reductive $X$-group scheme
and $\cG$ be a principal $\bG$-bundle over $\P^1_X$. Suppose $\cG|_{\infty \times X}$ is Zarisky locally trivial.
Then $\cG$ is Zarisky locally trivial on $\P^1_X$.
\end{thm}

If $X=\spec k$, where $k$ is a field, and $\bG$ is a reductive $k$-group,
principal $\bG$-bundles over $\P^1_k$ are very well
understood. Specifically,
 any principal $\bG$-bundle over $\P^1_k$ that has a trivial restriction to the infinity section,
is induced by a line bundle over $\P^1_k$ by means of a cocharacter $\mathbb G_{m,k}\to\bG$,
and any two such bundles are isomorphic if and
only if the corresponding cocharacters are permuted by the Weyl group of $\bG$.
In particular, every such principal $\bG$-bundle is trivial locally in the Zariski topology
on $\P^1_k$.
This theorem was proved by A. Grothendieck if $k$ is the field of complex numbers~\cite{Gr57}, by G. Harder if $\bG$ is quasi-split~\cite{Harder68},
and by P. Gille in general~\cite[Th. 3.8]{GilleTorseurs} (heavily relying on a theorem of M.~S. Raghunathan~\cite[Th. 3.4]{R1}).
These results are used in the present text in a substantial way.

{\it The following notation will be used from now on in the statements of our results.
}
Let $\cO$ be a semi-local  ring.
Put $U:=\spec \cO$.
Let $\bG$ be a reductive $U$-group scheme.
For a $U$-scheme $W$ we denote by $\mathbb A^1_W$ the affine line over $W$, and by $\P^1_W$ the projective line over $W$.
We identify ${\mathbb A}^1_W$ with $\P^1_W-\infty\times W$.
By a principal $\bG$-bundle over a $U$-scheme $W$ we understand a principal $\bG\times_U W$-bundle.

Theorem~\ref{th:X} is yielded by the following Theorem \ref{th-two-pullbacks}
 by means of several intermediate statements given below.

\begin{thm}\label{th-two-pullbacks}
For any reductive $U$-group scheme $\bG$, any
principal $\bG$-bundle $\cG$ over $\mathbb P^1_U$,
any semi-local  $U$-scheme $V$ and any two $U$-morphisms
$\phi_1,\phi_2: V\rightrightarrows \mathbb P^1_U$, the two principal $\bG$-bundles
$$\phi^*_1(\cG) \ \ \text{and} \ \ \phi^*_2(\cG)$$
over $V$ are isomorphic.
\end{thm}

\begin{cor}\label{cor:red-sections}
Let $U$, $\bG$ and $\cG$ be as in Theorem~\ref{th-two-pullbacks}.

$(i)$ For any two sections $s_1,s_2:U\to\P^1_U$ of the projection
$\P^1_U\to U$, the pull-backs $s_1^*(\cG)$ and $s_2^*(\cG)$ are isomorphic as principal $\bG$-bundles over $U$.

$(ii)$ If $\cG|_{\infty\times U}$ is trivial, then for any section $s:U\to\P^1_U$ of the projection
$\P^1_U\to U$, the pull-back $s^*(\cG)$ is trivial.
\end{cor}
After the initial version of the present paper was published on
arXiv.org, the statement of Corollary \ref{cor:red-sections}
was reproved and slightly generalized by
K. \v{C}esnavi\v{c}ius and R. Fedorov in
\cite[Theorem 3.6]{CF}. \\\\
Theorem~\ref{th:X} follows easily from Theorem \ref{th:zar-cover-red}
which is another corollary of
Theorem~\ref{th-two-pullbacks}.

\begin{thm}\label{th:zar-cover-red}
Let $U$, $\bG$ be as in Theorem~\ref{th-two-pullbacks}.

$(i)$
For any principal $\bG$-bundle $\cG$ over $\mathbb P^1_U$, there are two closed subschemes $Z_1$ and $Z_2$ of $\P^1_U$,
finite over $U$, such that $Z_1\cap Z_2=\emptyset$ and the $\bG$-bundles $\cG|_{\P^1_U-Z_1}$ and $\cG|_{\P^1_U-Z_2}$
are extended from $U$.
Moreover, if $E=\cG|_{\infty\times U}$, then these two bundles are both extensions of the $\bG$-bundle $E$.

$(ii)$ In particular, if $\cG|_{\infty\times U}$ is trivial,
then there are two closed subschemes $Z_1,Z_2$ as in $(i)$ such that $\cG$ is trivial away from $Z_i$ for $i=1,2$.
That is, our
$\cG$ is obtained by patching
two trivial $\bG$-bundles
over  $\Pro^1_U-(Z_1\cup Z_2)$.
\end{thm}

{\it In the semi-simple simply connected case we prove more precise results.
}
To state them recall that an $R$-group scheme $\bG$ is called simple
(respectively, semi-simple), if it is affine and smooth
as an $R$-scheme and if, moreover, for each ring homomorphism
$s:R\to\Omega(s)$ to an algebraically closed field $\Omega(s)$,
its scalar extension $\bG_{\Omega(s)}$
is a connected simple (respectively, semi-simple) algebraic
group over $\Omega(s)$. The class of semi-simple group schemes
contains the class of simple group schemes. This notion of a
simple $R$-group scheme coincides with the notion of a {simple
semi-simple $R$-group scheme from Demazure---Grothendieck
\cite[Exp.~XIX, D\'ef.~2.7; Exp.~XXIV, 5.3]{SGA3}.}

A semi-simple $R$-group scheme $\bG$ is called {\it simply connected\/}
provided that for each ring homomorphism
$s:R\to \Omega(s)$ of $R$ to an algebraically closed field
$\Omega(s)$ the scalar extension $\bG_{\Omega(s)}$ is a
simply connected $\Omega(s)$-group scheme.
This definition coincides with~\cite[Exp.~XXII, D\'ef.~4.3.3]{SGA3}.

\begin{thm}\label{cor:inf_1}
Let $U$ be semi-local and $\bG$ be a semi-simple simply connected $U$-group scheme.
Then there are two closed subschemes $Y'$ and $Y''$ in $\mathbb G_{m,U}$,  finite and \'{e}tale
over $U$ such that $Y'\cap Y''=\emptyset$ and
the set of
isomorphism classes of principal $\bG$-bundles
$\cG$ over $\P^1_U$ trivial at infinity
equals to the double coset
$$
\bG(\P^1_U-Y')\setminus \bG(\Pro^1_U-(Y'\cup Y''))\ /\ \bG(\P^1_U-Y'').
$$
\end{thm}

This theorem is a direct consequence of the following one
\begin{thm}\label{th:inf}
Let $U$ be as above and let $\bG$ be a semi-simple simply connected group scheme over $U$.

$(i)$ There is a closed subscheme $Y$ in $\mathbb A^1_U$,  finite and \'etale over $U$,
such that
every $\bG$-bundle $\cG$ on $\P^1_U$ that is trivial at infinity, is trivial also over $Y$ and over
$\P^1_U-Y$. If $Z\subset \P^1_U$ is a closed subscheme finite over $U$, then one can choose the above $Y$
such that additionally $Y\cap Z=\emptyset$.

$(ii)$
The isomorphism classes of principal $\bG$-bundles
$\cG$ over $\P^1_U$ such that $\cG|_{\infty\times U}$ is trivial are in bijective correspondence with
the double cosets
$$
\bG(\P^1_U-Y)\setminus \bG(\dot{Y}^h)\ /\ \bG(Y^h),
$$
where $Y^h$ stands for the henselization of the pair $(\mathbb A^1_U,Y)$ (see~\cite{Gabber} or~\cite{FP} for the definition) and
$\dot{Y}^h$ stands for $Y^h-Y$.
\end{thm}

The major difficulty in proving this Theorem is to check the triviality on $Y$
of any bundle $\cG$ as in the item $(i)$. The main novelty is that we use
Theorem \ref{th-14.03} below to check this (see the base change diagram~\eqref{base-change-via-Y} in the proof of
Theorem~\ref{th:sc}).

We use the following definition of an isotropic group.
Let $\bG$ be a semi-simple simply connected group scheme over a ring $R$.
Its parabolic $R$-subgroup scheme $\bP$ is called {\it strictly proper}, if for each ring homomorphism
$s:R\to\Omega(s)$, where
$\Omega(s)$ is an algebraically closed field,
the type of the parabolic subgroup $\bP_{\Omega(s)}$ in the sense of~\cite[Exp. XXVI, \S 3.2]{SGA3}
does not contain any connected component of the Dynkin diagram of $\bG_{\Omega(s)}$.
We will say that
a semi-simple simply connected group scheme
$\bG$ is {\it isotropic}, if it contains a strictly proper parabolic $R$-subgroup $\bP$.
For isotropic simply connected groups, we prove the following result
which parallels the classic Horroks' theorem on projective modules~\cite{Horrocks}.
Note that
{\it this result is crucial for the semi-simple simply connected case}.

\begin{thm}\label{th-14.03}
Let $W$ be an affine
 scheme.
Let $\bG$ be an isotropic simply connected
semi-simple
$W$-group scheme.
Let $\cG$ be a principal $\bG$-bundle over $\mathbb P^1_W$ such that
$\mathcal G|_{\infty\times W}$ is a trivial $\bG$-bundle. Then $\cG|_{\mathbb A^1_W}$ is trivial.
\end{thm}

The idea of the proof of Theorem~\ref{th-two-pullbacks} is to reduce the case of a general
reductive group to the two special cases: the one of a torus, where one uses standard spectral sequence
arguments, and the one of a simply connected group, where one applies Theorem~\ref{th:inf}.
The reduction of the general case to these two cases uses a version of the trick that first appeared in the proof of~\cite[Theorem 4]{Fed},
namely, the application of a suitable $n$th power endomorphism of $\P^1_U$.

The present text is organized as follows.
In Section \ref{isotr_section} Theorem \ref{th-14.03} is proved.
In Section \ref{SchemeY_section} we prove a crucial special case of Theorem~\ref{th:inf}.
In Section \ref{semi-simple} Theorem~\ref{th:inf} and Theorem~\ref{cor:inf_1} are proved.
In Section~\ref{sect:t_sect} we prove some results on toral bundles which are necessary to prove
Theorem~\ref{th-two-pullbacks}.
In Section~\ref{sec:red} we prove Theorem~\ref{th-two-pullbacks} and deduce from it
Theorems~\ref{th:zar-cover-red} and~\ref{th:X}.

We thank K\c{e}stutis \v{C}esnavi\v{c}ius and Roman Fedorov for bringing to our attention that the regularity assumption on $U$
is not necessary.

Finally, indicate briefly a relationship of the topic of the present text to
the well-known conjecture due to J.-P.~Serre and A.~Grothendieck. They conjectured
(see~\cite[Remarque, p.31]{Se}, \cite[Remarque 3, p.26-27]{Gr1}, and~\cite[Remarque~1.11.a]{Gr2})
that if $X$ is a regular scheme, then any generically trivial principal bundle
$\cG$ over $X$ is trivial locally in the Zariski topology.
Recall that if $X$ is over a field, then the conjecture is proved in \cite{Pan3}.
If $X$ is over an {\it an infinite} field, then the conjecture is proved earlier in \cite{FP}.
A detailed description of earlier results can be found in~\cite{P}.
Most of results on the Grothendieck--Serre conjecture rely in their proofs on a
suitable theorem on the structure of principal $\bG$-bundles on the relative projective line.
The results of the present paper are used by the authors
in the proof of the ``constant'' mixed characteristic case of the Grothendieck--Serre conjecture~\cite{PSt},
and are expected to be useful in obtaining further results on this conjecture.

\section{Proof of the isotropic simply connected case}\label{isotr_section}
The main aim of this section is to prove Theorem \ref{th-14.03} on isotropic simply connected groups.
We use in the proof the following two statements on not necessarily isotropic groups extending~\cite[Lemma A.1 and Proposition 9.6]{PSV}.

\begin{lem}
\label{Horrocks}
Let $W$ be an affine
 scheme.
Let $\bH$ be a simply connected semi-simple $W$-group scheme and $\bH'=\text{GL}_{N,W}$.
Let $j: \bH \hra \bH'$ be a closed embedding of the $W$-group scheme.

Let $F \in H^1(\Pro^1_W, \bH)$
be a principal $\bH$-bundle, and let
$M :=j_*(F) \in H^1(\Pro^1_W, \bH')$
be the corresponding principal $H^{\prime}$-bundle.
If $M$ is a trivial $\bH^{\prime}$-bundle, then there exists a principal
$\bH$-bundle $F_0$ over $W$ such that
$pr^*(F_0) \cong F$,
where
$pr: \Pro^1_W \to W$
is the canonical projection.
\end{lem}

\begin{proof}
By \cite[Cor.6.12.(ii)]{C-T-S2} the $fppf$-sheaf
$\bH'/j(\bH)$ is representable by an affine $W$-scheme $X$ of finite type over $W$.
Consider the long exact sequence of pointed sets
$$1 \to \bH(\Pro^1_W) \xra{j_*} \bH^{\prime}(\Pro^1_W) \to X(\Pro^1_W)
\xra{\partial} H^1_{\text{\'{e}t}}(\Pro^1_W, \bH) \xra{j_*} H^1_{\text{\'{e}t}}(\Pro^1_W, \bH^{\prime}).$$
Since $j_*(F)$ is trivial, there is $\varphi \in X(\Pro^1_W)$ such that
$\partial (\varphi)= F$.

The $W$-morphism $\varphi: \Pro^1_W \to X$ is a $W$-morphism of a $W$-projective scheme to a $W$-affine scheme $X$ of finite type.
Thus $\varphi$ is "constant", that is, there exists a section $s: W \to X$ such that
$\varphi= s \circ pr$.
Consider another long exact sequence of pointed sets, this time the one corresponding to the scheme $W$,
and the morphism of the first sequence to the second one induced by the projection $pr$. We get a big commutative diagram.
In particular, we get the following commutative square
\begin{equation}
\label{Representation1}
    \xymatrix{
X(W) \ar[d]_{pr^*_W} \ar[rr]^{\partial} && H^1_{\text{\'{e}t}}(W, H) \ar[d]^{pr^*_W} \\
X(\Pro^1_W) \ar[rr]^{\partial} && H^1_{\text{\'{e}t}}(\Pro^1_W, H).       \\  }
\end{equation}
We have $\pr^*_W(s)=\varphi$. Hence
$$F= \partial (\varphi) = \partial (\pr^*_W(s)) = pr^*_W ( \partial (s)).$$
Setting $F_0=\partial (s)$ we see that
$F= pr^*_W ( F_0)$. The Lemma is proved.

\end{proof}
Now let $\cO'$ be a semi-local  ring with and put $W:=\spec \cO'$.
Let $w_1,\ldots,w_m$ be the closed points of $W$. Let $k(w_i)$ be the residue field of $w_i$.
Consider the reduced closed subscheme $\bw$ of $W$
whose points are $w_1$, \ldots, $w_v$. Thus
$\bw\cong\coprod_i\spec k(w_i)$.
We are going to use the following theorem, which is an extension of
\cite{PSV,Tsy}.
\begin{thm}\label{th:tsy}
Let $W=\spec \cO'$ be semi-local  scheme.
Let $\bG$ be a semi-simple $W$-group scheme. Let $\cG$ be a principal $\bG$-bundle over $\P^1_{W}$.
Let $\cG_\bw$ be the restriction of $\cG$ to $\P^1_{\bw}$. If $\cG_\bw$ is a trivial principal $\bG$-bundle, then
$\cG$ is extended from $W$ via the pull-back along the canonical projection $\P^1_{W}\to W$.
\end{thm}
\begin{proof}
If $\cO'$ contains a field, then $W$ is a  semi-local scheme
over the prime subfield of $\cO'$, and the claim is precisely~\cite[Proposition 9.6]{PSV}.

In general, by \cite[Corollary 3.2]{Thomason}
there is a closed $W$-group scheme embedding
$j: \bG \hra \text{GL}_{N,W}$
for an $N > 0$. By the assumption of the theorem, the $\bG$-bundle $\cG$ is trivial on
$\Pro^1_{\bw} \subset \Pro^1_W$. Hence the $\text{GL}_{N,W}$-bundle $j_*(\cG)$ over $\Pro^1_W$
is trivial over $\Pro^1_{\bw}$.

The $W$-group scheme
$\text{GL}_{N,W}$
is just the ordinary general linear group.
Thus $j_*(\cG)$ corresponds to a vector bundle $M$ over $\Pro^1_W$.
Moreover, this vector bundle is trivial on
$\Pro^1_{\bw}$.
Using the equality
$H^1(\Pro^1, \mathcal O_{\Pro^1})=0$
and
\cite[Cor. 4.6.4]{EGAIII1},
we see that $M$ is of the form $M=pr^*(M_0)$
for a vector bundle $M_0$ over $W$.
Since $W$ is semi-local, $M_0$ is trivial over $W$.
Thus $M$ is trivial on $\Pro^1_W$.
Thus
$j_*(\cG)$ is a trivial $\text{GL}_{N,W}$-bundle.

Now, applying Lemma~\ref{Horrocks} to the embedding
$j:\bG \hra \text{GL}_{N,W}$,
we see that $\cG=pr^*(\cG_0)$ for some
$\cG_0 \in H^1(W, \bG)$.
Theorem~\ref{th:tsy} is proved.
\end{proof}

In what follows we use the notion of an isotropic simply connected semisimple group introduced in the Introduction.
The following lemma is in order

\begin{lem}\label{strictly proper_parabolic}
Let $\tilde W$ be a connected affine scheme, which is finite \'etale over a connected affine scheme $W$, and let
$\tilde \bG$ is a simple simply connected $\tilde W$-group scheme.

$(i)$ If $\tilde \bG$ is isotropic and $\tilde \bP\subset \tilde \bG$ is a strictly proper parabolic subgroup, then
the Weil restriction $R_{\tilde W/W}(\tilde \bP)$ is a strictly proper parabolic
in the Weil restriction $R_{\tilde W/W}(\tilde \bG)$, and thus $R_{\tilde W/W}(\tilde \bG)$ is isotropic.

$(ii)$ Assume that $W$ is, moreover, semi-local. If $R_{\tilde W/W}(\tilde \bG)$ has a proper parabolic $W$-subgroup,
then $\tilde \bG$ is isotropic, and, as a consequence, $R_{\tilde W/W}(\tilde \bG)$ is isotropic and has a strictly proper
parabolic subgroup.
\end{lem}
\begin{proof}
$(i)$ Write $W=\spec R$ for a connected commutative ring $R$, and take a ring homomorphism $s:R\to \Omega(s)$ from $R$ to an algebraically closed field
$\Omega(s)$. Then $\tilde W\times_W \spec \Omega(s)$ is isomorphic to the disjoint union
of $[\tilde W:W]$ copies of $\spec\Omega(s)$, and $R_{\tilde W/W}(\tilde \bG)\times_W\spec \Omega(s)$ is isomorphic to the product
of $[\tilde W:W]$ copies of the simple simply connected algebraic $\Omega(s)$-group $\tilde\bG_{\Omega(s)}$.
Consequently, the Dynkin diagram of $R_{\tilde W/W}(\tilde \bG)\times_W\spec \Omega(s)$
is a disjoint union of the same number of copies of the Dynkin diagram of $\tilde\bG_{\Omega(s)}$.
Similarly, $R_{\tilde W/W}(\tilde \bP)\times_W\spec \Omega(s)$ is isomorphic to the product
of $[\tilde W:W]$ copies of the parabolic subgroup $\tilde\bP_{\Omega(s)}$ of $\tilde\bG_{\Omega(s)}$.
Consequently, $R_{\tilde W/W}(\tilde \bP)\times_W\spec \Omega(s)$ is a strictly proper parabolic subgroup of
$R_{\tilde W/W}(\tilde \bG)\times_W\spec \Omega(s)$, i.e. its type does not contain any connected component of the Dynkin
diagram of this group.
Since $R_{\tilde W/W}(\tilde \bP)$ is also a smooth closed $W$-subgroup
of $R_{\tilde W/W}(\tilde \bG)$, by the very definition of a parabolic subgroup~\cite[Exp. XXVI, D\'ef. 1.1]{SGA3}
it is a parabolic $W$-subgroup, and, clearly, it is strictly proper.

$(ii)$ Assume that $W$ is semi-local. Then $\tilde W$ is also semi-local.
If $R_{\tilde W/W}(\tilde \bG)$ has a proper parabolic subgroup,
by~\cite[Exp. XXVI, Corollaire 6.12]{SGA3}
$R_{\tilde W/W}(\tilde \bG)$ contains a closed subgroup of the form $\mathbb G_{m,W}$.
By the adjunction between Weil restriction and base change,  it follows that
there is a non-trivial homomorphism
$\varphi: \mathbb{G}_{m,\tilde W} \to \tilde\bG$. Then again
by~\cite[Exp. XXVI, Corollaire 6.12]{SGA3} $\tilde G$ has a proper parabolic $\tilde W$-subgroup. Since $\tilde\bG$
is simple and $\tilde W$ is connected, this parabolic subgroup is necessarily strictly proper. Hence $\tilde\bG$ is
isotropic. It remains to apply $(i)$.
\end{proof}

\begin{lem}\label{lem:isotr}
Let $W$ be the semi-local  scheme and $\bw\subset W$ be its the above closed reduced subscheme.
Let $\bG$ be a simply connected semi-simple
$W$-group scheme. Let $Z\subseteq\mathbb{A}^1_W$ be a closed subscheme finite and \'etale over $W$, and
such that $\bG_Z$ is isotropic.
Assume also that there is a section $s:W\to\P^1_W$ of the canonical projection $\P^1_W\to W$ such that
$s(W)$ is a subscheme of either $Z$ or $\P^1_W-Z$.
Let $\cG$ be a principal $\bG$-bundle over $\mathbb P^1_W$ such that the restrictions
$\cG|_Z$, $\cG|_{s(W)}$ and $\cG_{\bw}|_{\P^1_\bw-Z_\bw}$ are trivial $\bG$-bundles.
Then $\cG|_{\P^1_W-Z}$ is trivial.
\end{lem}
\begin{proof}
By assumption the group $\bG_Z$ is isotropic. Let $\bP$ be a strictly proper parabolic $Z$-subgroup of $\bG_Z$.
For any affine $Z$-scheme $V$ denote by $\bE(\bG(V))$ the subgroup of $\bG(V)$ generated by the $V$-points
of the unipotent radical of $\bP$ and of the unipotent radical of any fixed parabolic $Z$-subgroup of $\bG_Z$
opposite to $\bP$ (such a parabolic subgroup exists by~\cite[Exp. XXVI, Cor. 2.3; Th. 4.3.2]{SGA3}).

Denote by $Z^h$ the henselization of $\mathbb A^1_W$ along $Z$ and by $\dot Z^h$ the open subscheme
$Z^h-Z$ of $Z^h$. By assumption, $\mathcal G|_Z$ is a trivial bundle.
As a consequence the $\bG$-bundle $\cG|_{Z^h}$ is trivial too.

Put $\cG'=\cG|_{\P^1_U - Z}$. Since the $\bG$-bundle $\cG|_{Z^h}$ is trivial
we can present $\cG$ in the form $\cE(\cG', \phi)$,
where $\phi: \dot Z^h\times_W \bG\to \cG|_{\dot Z^h}$
is a $\bG_{\dot Z^h}$-bundle isomorphism.
In this case the $\cG_\bw$ is presented in the form $\cE(\cG'_\bw, \phi_\bw)$,
where $\phi_\bw: \dot Z^h_\bw\times_\bw \bG_{\bw}\to \cG_\bw|_{\dot Z^h_\bw}$
is a $\bG_{\dot Z^h_\bw}$-bundle isomorphism.

By assumption the bundle $\cG_\bw|_{\P^1_\bw-Z_\bw}$
is trivial.
Using arguments as in \cite{FP}
one can find an element $\alpha_\bw\in \bE(\bG_\bw(\dot Z^h_\bw))$
such that the $\bG_\bw$-bundle $\E(\cG'_\bw, \phi_\bw\circ \alpha_\bw)$ over $\P^1_\bw$
is trivial.
Namely, since $\bP_{Z^h_\bw}$ is strictly proper,
 by~\cite[Fait 4.3, Lemme 4.5]{Gille:BourbakiTalk} one has
$$
\bG(\dot Z^h_\bw)=\bG(Z^h_\bw)\cdot \bE(\bG_\bw(\dot Z^h_\bw)).
$$

The scheme $\dot Z^h$ is affine. Thus, the group homomorphism
$$
\bE(\bG(\dot Z^h))\to \bE(\bG_\bw(\dot Z^h_\bw))
$$
is surjective.
Thus, there is an element
$\alpha\in \bG(\dot Z^h)$
such that
$\alpha|_{\dot Z^h_\bw}=\alpha_\bw$ in $\bG_\bw(\dot Z^h_\bw)$.
Put
$$
\cG_{mod}=\cE(\cG', \phi\circ \alpha).
$$
Then the $\bG_\bw$-bundle
$(\cG_{mod})_\bw=\E(\cG'_\bw, \phi_\bw\circ \alpha_\bw)$ is trivial over $\P^1_\bw$.

Since the bundle
$(\cG_{mod})_\bw$ is trivial, by Theorem~\ref{th:tsy}
the $\bG$-bundle $\cG_{mod}$
is extended from $W$ via the pull-back along the projection
$pr_W: \P^1_W\to W$.

On the other hand $\cG_{mod}|_{s(W)}$
is trivial, since $s(W)\subseteq \P^1_W-Z$ or $s(W)\subseteq Z$ and
$\cG_{mod}|_{\P^1_W-Z}=\cG|_{\P^1_W-Z}$ and $\cG_{mod}|_Z=\cG|_Z$.
Hence the $\bG$-bundle $\cG_{mod}$ is trivial over $\P^1_W$.
Since $\cG|_{\P^1_W - Z}=\cG_{mod}|_{\P^1_W - Z}$
it follows that
the $\bG$-bundle $\cG|_{\P^1_W - Z}$ is trivial.

\end{proof}

\begin{thm}\label{th-14.03-simple}
Let $W$ be a connected semi-local  scheme.
Let $\bG$ be an isotropic simply connected simple
$W$-group scheme.
Let $\cG$ be a principal $\bG$-bundle over $\mathbb P^1_W$ such that
$\mathcal G|_{\infty\times W}$ is a trivial $\bG$-bundle. Then $\cG|_{\mathbb A^1_W}$ is trivial.
\end{thm}
\begin{proof}
We apply Lemma~\ref{lem:isotr} to the given group scheme $\bG$ and the closed subscheme $Z=\infty\times W$ of the
affine line $\P^1_W-(0\times W)$.
We take $s:W\to\P^1_W$ to be the
natural isomorphism between $W$ and $\infty\times W=Z$. By assumption of the theorem, $\cG|_{s(W)}$ is trivial.
Also by assumption $\bG$ is an isotropic simply connected simple $W$-group scheme.
Hence $\bG_Z$ satisfies the condition of Lemma~\ref{lem:isotr}.
Since $\P^1_W-Z={\mathbb A}^1_W$, we have that $\text{Pic}(\P^1_\bw-Z_\bw)$ is trivial.
Then by the Gille Theorem \cite{GilleTorseurs} $\cG_\bw|_{\P^1_\bw-Z_\bw}$
is trivial.  Then by Lemma~\ref{lem:isotr} $\cG|_{\P^1_W - Z}=\cG|_{{\mathbb A}^1_W}$ is trivial.
\end{proof}

\begin{proof}[Proof of Theorem \ref{th-14.03}]
First, we may and will suppose that the scheme $W$ is connected.
Second, by a generalization of Quillen's local-global principle~\cite[Theorem 3.2.5]{AHW}
we may and will suppose that the scheme $W$ is local (or connected semi-local).
Applying the Faddeev--Shapiro Lemma \cite[Exp. XXIV Prop. 8.4]{SGA3} and Lemma~\ref{strictly proper_parabolic}
we may and will suppose that
the scheme $W$ is connected semi-local and the $W$-group scheme $\bG$
is an isotropic simply connected simple $W$-group scheme.
In this case Theorem \ref{th-14.03} coincides with
Theorem \ref{th-14.03-simple}, which is already proved.
\end{proof}

\section{A special case of Theorem~\ref{th:inf} }\label{SchemeY_section}

The aim of this Section is to prove a special case of Theorem~\ref{th:inf} that will be used to establish the general case.
This is done in Theorem~\ref{th:sc} in the end of this Section.

Let $\cO$ be a semi-local  ring and set $U:=\spec \cO$.
Let $u_1,\ldots,u_n$ be the closed points of $U$. Let $k(u_i)$ be the residue field of $u_i$.
Consider the reduced closed subscheme $\bu$ of $U$
whose points are $u_1$, \ldots, $u_n$. Thus
$ \bu\cong\coprod_i\spec k(u_i)$.

Let $\bG$ be a simply connected
semi-simple
$U$-group scheme.
By $\bG_{u_i}$ we denote the fibre of $\bG$ over $u_i$.
Let $\bu'\subset\bu$ be the subscheme of all closed points $u\in \bu$ such that the simply connected $k(u)$-group
$\bG_{u}$ is isotropic. Set $\bu''=\bu-\bu'$.
It is possible that $\bu'$ or $\bu''$ is empty.

\begin{prop}
\label{SchemeY22}
Suppose $U$ is connected.
Let $\bG$ be a simply connected
semi-simple
$U$-group scheme of the form $R_{\tilde U/U}(\tilde \bG)$,
where $\tilde U$ is finite \'etale over $U$ and connected,
and the $\tilde U$-group scheme $\tilde \bG$ is simple.
Suppose that $\bu'$ is non-empty. Then the following is true.\\
(i) There is a $U$-scheme $Y'$ finite and \'etale over $U$ and such that \\
\indent (a) the $Y'$-group scheme $\bG_{Y'}:=\bG\times_U Y'$ is isotropic;\\
\indent (b) there is a section $s': \bu'\to Y'$ of $Y'$ over the $\bu'$.\\
(ii) For each $U$-scheme $Y'$ as in the item (i) there is a diagram of the form
$$Y'\xleftarrow{\pi} Y\xrightarrow{j} \mathbb A^1_U$$
with a finite \'{e}tale morphism $\pi$ and a closed $U$-embedding $j$
such that for each $u\in \bu'$ one has $1=\text{g.c.d.}_{v\in Y_u}[k(v):k(u)]$. \\
(iii) Let $Z\subset \P^1_U$ be a closed subscheme finite over $U$.
Then
for each $U$-scheme $Y'$ as in the item (i)
there is a diagram $Y'\xleftarrow{\pi} Y\xrightarrow{j} \mathbb A^1_U$ as in the item (ii) such that
$j(Y)\cap Z=\emptyset$.
\end{prop}

\begin{rem}\label{Y'_over_U'_is a_Nis_neiborh}
The item (i) of the Proposition one can read as follows:
there is a Nisnevich neiborhood $Y'/U$ of the $\bu'$ which makes the
group scheme $\bG$ isotropic and moreover which is finite \'{e}tale over $U$.

If all the residue fields of the scheme $U$ are infinite, then there is an essential simplification
of the item (ii). Namely, in this case there is a closed $U$-embedding $j: Y'\hra \mathbb A^1_U$.
Clearly, $j\circ s': \bu' \to \mathbb A^1_U$ is a section of the projection
$\mathbb A^1_U\to U$ over $\bu'$.

Furthermore in the infinite residue field case the property $j(Y)\cap Z=\emptyset$ can be achieved just
by composing $j$ with an affine translation of the affine line $\mathbb A^1_U$.
\end{rem}

The main aim of this section is to prove this Proposition.
We begin with a general lemma going back to \cite[Lemma 7.2]{OP2}.
\begin{lem}\label{Lemma7_2}
Let $S=\spec R$ be a
semi-local
scheme.
Let $T$ be a closed subscheme of $S$. Let
$\bar {\cX}$
be a closed subscheme of
$\P^d_S=Proj (R[X_0,...,X_d])$
and let
$\cX = \bar {\cX} \cap \mathbb A^d_S$, where
$\mathbb A^d_S$
is the affine space defined by $X_0\neq 0$. Let $\cX_{\infty} = \bar {\cX} - \cX$ be the intersection of
$\bar {\cX}$
with the hyperplane at infinity $X_0 = 0$.
Assume that over T there exists a section
$\delta: T \to \cX$ of the canonical projection $p: \cX \to S$. Assume further that
\begin{itemize}
\item[(1)]$\mathcal X$ is $S$-smooth and equidimensional over $S$,
of relative dimension $r$;
\item[(2)]
for every closed
point $s\in S$ the closed fibres of $\mathcal X_\infty$ and $\mathcal X$
satisfy
$$\dim (\mathcal X_\infty(s))< \dim (\mathcal X(s))=r\;.$$
\item[(3)]
Over $T$ there exists a section $\delta:T\to \mathcal X$
of the morphism $p: \mathcal X\to S$.
\end{itemize}
Then there is a closed subscheme $\tilde S$ of $\mathcal X$ which is finite
\'etale over $S$ and contains $\delta(T)$.
\end{lem}

\begin{proof}
It is clear that the statement of the lemma can be proved independently for every connected component of $S$, so we assume that
$S$ is connected. For a connected semi-local scheme $S$,
the same statement was proved in~\cite[Lemma 4.3]{Pan1} under additional assumptions that $S$ was regular,
all residue fields of $S$ were finite, and $T$ was connected.
The fact that some residue fields of $S$ may be infinite does not change the proof in any way.
The regularity of $S$ was used at one point only, namely, in the reference to~\cite[Lemma 7.3]{OP2}, and the latter
reference can be replaced by the reference to~\cite[Th\'eor\`eme 11.3.8]{EGAIV3}.
The fact that $T$ is not required to be connected is taken into account by not requiring $\tilde S$
to be connected in the conclusion of the lemma. Apart from that, all arguments proving \cite[Lemma 4.3]{Pan1}
carry over to the present situation verbatim.
\end{proof}

The following result is a slight extension of
\cite[Lemma 3.1]{Pan3}.

\begin{lem}\label{OjPan}
Let $S=\spec R$ be a
semi-local
scheme.
Let $T$ be a closed
subscheme of $S$. Let $W$ be a closed subscheme of the projective space
$\P^n_S$.
Assume that over $T$ there exists a section
$\delta:T\to W$
of the canonical projection $W\to S$. Assume further that
$W$ is smooth and equidimensional over $S$,
of relative dimension $r$.
Then there exists a closed subscheme $\tilde S$ of $W$ which is finite
\'etale over
$S$ and contains $\delta(T)$.
\end{lem}

\begin{proof}[Proof of Lemma \ref{OjPan}]
Formally this lemma is a bit different from Lemma \ref{Lemma7_2}.
However it can be derived from Lemma \ref{Lemma7_2}
as follows.
There is a Veronese embedding $\Bbb P^n_S\hookrightarrow \Bbb P^d_S$
and linear homogeneous coordinates $X_0,X_1,...,X_d$ on $\Bbb P^d_S$
such that the following holds:\\
\smallskip
{\rm{(1)}}
if
$W_{\infty}=W\cap \{X_0=0\}$ is the intersection of $W$ with the
hyperplane
$X_0=0$, then for every closed
point $s\in S$ the closed fibres of $W_\infty$
satisfy
$\text{dim}(W_{\infty}(s))<r=\text{dim}(W(s))$; \\
{\rm{(2)}}
$\delta$ maps $T$ into
the closed subscheme of $\Bbb P^d_T$ defined by
$X_1=\dots=X_d=0$.\\
\smallskip
Applying now Lemma \ref{Lemma7_2}
to $\bar{\mathcal X}=W$, $\mathcal X_{\infty}=W_{\infty}$ and $\mathcal X=W-W_{\infty}$
we get a closed subscheme $\tilde S$ of $W-W_{\infty}$ which is finite
\'etale over
$S$ and contains $\delta(T)$. Since $\tilde S$ is finite over $S$ it is closed in $W$ as well.
\end{proof}

\begin{proof}[Proof of Proposition \ref{SchemeY22}]
First prove the item (i). We use the following notions introduced in~\cite[Exp. XXVI, \S 3]{SGA3}.
Let $q: \cP \to U$ be the $U$-scheme of all parabolic subgroup schemes of $\bG$.
Let $\pi: T\to U$ by the scheme of types of all parabolic subgroup schemes of $\bG$.
Let $\cP\xra{r} T\xra{\pi} U$ be the Stein decomposition of the morphism $q$~\cite[Exp. XXVI, Corollaire 3.5]{SGA3}.
Recall that $\pi$ is finite, \'{e}tale surjective morphism and $r$ is smooth projective
with connected fibres.

Write $\tilde \bu'$ for the pre-image of $\bu'$ in $\tilde U$.
Since $\bG_{\bu'}$ is isotropic it follows that $\tilde \bG_{\tilde \bu'}$ is isotropic
by Lemma~\ref{strictly proper_parabolic}.
Thus, there is a strictly proper parabolic $\tilde \bP'_{\tilde \bu'}$
in the simple $\tilde \bu'$-group scheme $\tilde \bG_{\tilde \bu'}$.
Put $\bP_{\bu'}=R_{\tilde \bu'/\bu'}(\tilde \bP'_{\tilde \bu'})$.
Lemma \ref{strictly proper_parabolic} yields that it is a strictly proper parabolic in
$\bG_{\bu'}=R_{\tilde \bu'/\bu'}(\tilde \bG'_{\tilde \bu'})$.

The parabolic $\bP'_{\bu'}$ in $\bG_{\bu'}$ is
a section $\delta: \bu' \to \cP$ of the morphism $q$ over the $\bu'$. Thus,
$r\circ \delta$ is a section of the morphism $\pi$ over the $\bu'$.
In particular, the morphism $t:=r\circ \delta: \bu' \to T$ is a closed embedding.

For each point $u_i$ in $\bu'$ put $p_i=\delta(u_i)$ and $t_i=t(u_i)$.
Let $T'_i\subset T$ be the connected component of $T$ containing the closed point $t_i$.
Put $\cP_i=r^{-1}(T_i)\subset \cP$ and
$r_i:=r|_{\cP_i}: \cP_i \to T_i$.
The morphism
$$r_i: \cP_i \to T_i$$
is smooth projective  over $T_i$
(see~\cite[Cor.~3.5, Exp.~XXVI]{SGA3}), and hence equidimensional.

By Lemma~\ref{OjPan}
we can find a closed subscheme $Y'_i\subset \cP_i$ such that $Y'_i$ is finite \'etale over $T_i$
and $p_i=[\delta\circ \pi](t_i)\subset Y'_i$.
That is $\delta(u_i)$ is in $Y'_i$.
For each $Y'_i\subset \cP_i$ let $Y_i$ be its unique connected component containing the point $p_i$.
Recall that $\bP_{\bu'}=R_{\tilde \bu'/\bu'}(\tilde \bP'_{\tilde \bu'})$ is a strictly proper parabolic in
$\bG_{\bu'}=R_{\tilde \bu'/\bu'}(\tilde \bG'_{\tilde \bu'})$. This yields the following \\
{\it Claim}.
The $Y_i$-group scheme $\bG_{Y_i}$ is isotropic.\\
Prove this Claim. Let $\bP(Y_i)$ be the parabolic subgroup scheme of $\bG_{Y_i}$
corresponding to the closed subscheme $Y_i$ of $\cP$.
Since $Y_i$ is connected, by~\cite[Exp. XXII, Proposition 2.8, Lemme 5.2.7]{SGA3} it suffices to check that at the point $p_i$
the parabolic $\bP(Y_i)_{p_i}$ in $(\bG_{Y_i})_{p_i}=\bG_{p_i}$
is a strictly proper parabolic.
Since $p_i=\delta(u_i)$,
we know that at the point $p_i$ the group equals to
$R_{\tilde \bu_i/u_i}(\tilde \bG_{\tilde \bu_i})$
and its parabolic equals
$R_{\tilde \bu_i/u_i}(\tilde \bP_{\tilde \bu_i})$.
By Lemma \ref{strictly proper_parabolic} this parabolic is
a strictly proper parabolic
of the group
$R_{\tilde \bu_i/u_i}(\tilde \bG_{\tilde \bu_i})$.
This proves the Claim.

The scheme $Y_i$ is finite, \'{e}tale over $U$ and
$\delta: u_i \to p_i\in Y_i$ is a point of $Y_i$.
Put $Y'=\sqcup_{u_i\in \bu'}Y_i$.
By the above Claim the $Y'$-group scheme $\bG_{Y'}:=\bG\times_U Y'$ is isotropic.
By construction $Y'$ is finite and \'etale over $U$, and provided with the section
$\delta: \bu' \to Y'$. The item (i) of Proposition is proved.

It remains to prove the items (ii) and (iii).
It is done in the Appendix. Specifically, as explained in Remark \ref{222yields22},
Proposition \ref{SchemeY222} applies to the situation of Proposition \ref{SchemeY22}, and yields
the remaining claims. Proposition \ref{SchemeY222} itself is proved in the Appendix.

\end{proof}

\begin{thm}\label{th:sc}
Suppose $U$ is connected.
Let $\bG$ be a simply connected
semi-simple
$U$-group scheme of the form $R_{\tilde U/U}(\tilde \bG)$,
where $\tilde U$ is finite \'etale over $U$ and connected,
and the $\tilde U$-group scheme $\tilde \bG$ is simple.
Let $\cG$ be a principal $\bG$-bundle
over $\P^1_U$ whose restriction to $\infty \times U$ is trivial. Then the following is true. \\
(i) If $\bG$ is anisotropic over every $u\in\bu$ (i.e. $\bu'$ is empty), then the $\bG$-bundle $\cG$ is trivial.
\\
(ii) If $\bu'$ is non-empty,
then there are a scheme $Y'$ and a diagram
$Y'\xleftarrow{\pi} Y\xrightarrow{j} \mathbb A^1_U$
as in the item (ii) of Proposition \ref{SchemeY22}.
For any such $Y'$ and $(\pi,j)$
the restrictions $\cG|_{\P^1_U-j(Y)}$ and $\cG|_{j(Y)}$ of the $\bG$-bundle $\cG$ are trivial
($j(Y)$ is closed in $\P^1_U$ since $Y$ is finite over $U$).\\
(iii) Assume that $\bu'$ is non-empty, and let $Z\subset \P^1_U$ be a closed subscheme finite over $U$.
Then there are $Y'$
and a diagram
$Y'\xleftarrow{\pi} Y\xrightarrow{j} \mathbb A^1_U$
as
in the item (ii) of this theorem such that $j(Y)\cap Z=\emptyset$.
\end{thm}

\begin{proof}[Proof of Theorem \ref{th:sc}]

If $\bu'$ is empty (i.e. $\bG$ is anisotropic over $\bu$), then the bundle $\cG_{\bu}$ is trivial by
the Gille theorem
\cite[Th\'eor\`eme 3.8]{GilleTorseurs}.
Since the bundle
$\cG_\bu$ is trivial, by Theorem~\ref{th:tsy}
the $\bG$-bundle $\cG$
is extended from $U$ via the pull-back along the projection
$pr_U: \P^1_U\to U$. By assumption $\cG|_{\infty\times U}$ is trivial, hence $\cG$ is trivial.

Suppose $\bu'$ is non-empty.
Take a finite \'etale $U$-scheme $Y$ as in Proposition \ref{SchemeY22}.
We apply Theorem~\ref{th-14.03} to the scheme $Y$, the $Y$-group scheme
$\bG_Y$ and $\mathcal G_Y=\cG \times_U Y$ over $\mathbb P^1_Y=\mathbb P^1_U\times_U Y$. By the item (i)
of Proposition \ref{SchemeY22}
the $Y$-group scheme $\bG_Y$ is isotropic and the $\bG_Y$-bundle $\mathcal G_Y$ over $\mathbb P^1_Y$
is such that
its restriction to $\infty \times Y$ is trivial.
By Theorem~\ref{th-14.03} the bundle
 $\cG_{Y}|_{\mathbb A^1_Y}=(\cG\times_U Y)|_{\mathbb A^1_U \times_U Y}$ is trivial.

Recall that by Proposition \ref{SchemeY22} there is a closed $U$-embedding $j: Y \hra \mathbb A^1_U$.
Consider the commutative diagram
\begin{equation}\label{base-change-via-Y}
    \xymatrix{
Y \ar[drr]_{j}\ar[rr]^{(j,id_Y)} && \mathbb A^1_U \times_U Y \ar[d]_{p} \ar[rr]^{Inc} && \P^1_U \times_U Y \ar[d]_{P}\\
&& \mathbb A^1_U \ar[rr]^{inc} && \P^1_U \\
}
\end{equation}
The $\bG_Y$-bundle $\mathcal G_Y|_{\mathbb A^1_Y}=(P\circ Inc)^*(\mathcal G)$ is trivial.
Thus the $\bG$-bundle
$$
\cG|_{j(Y)}=(inc\circ j)^*(\cG)= (j,id_Y)^*((P\circ Inc)^*(\mathcal G))
$$
is trivial.

It remains to prove that $\cG|_{\P^1_U-j(Y)}$ is trivial.
We apply Lemma~\ref{lem:isotr} to the group scheme $\bG$, the closed subscheme $Z=j(Y)$ of $\P^1_U$, and the section
$s:U\to\P^1_U$ given by the natural isomorphism between $U$ and $\infty\times U\subseteq\P^1_U-j(Y)$. Let us check
all the conditions of that Lemma.
The group $\bG|_{j(Y)}$ is isotropic, since $\bG_Y$ is isotropic by the choice of $Y$.
We have proved above that $\cG|_{j(Y)}$ is trivial.
Also, by Proposition \ref{SchemeY22} for
each $u\in \bu'$ one has $1=\text{g.c.d.}_{v\in Y_u}[k(v):k(u)]$.
Hence we have that
$\text{Pic}(\P^1_{\bu'}-j(Y)_{\bu'})$ is trivial. Then
by the Gille theorem \cite[Corollaire 3.10]{GilleTorseurs}
$\cG|_{\P^1_{\bu'}-j(Y)_{\bu'}}$ is trivial. On the other hand, for every $u\in\bu-\bu'$ the group
$\bG_u$ has no proper parabolic subgroups, since otherwise it would have been isotropic by Lemma~\ref{strictly
proper_parabolic}. Then by the same result of Gille the bundle $\cG|_{\P^1_{\bu-\bu'}}$ is trivial. Consequently,
$\cG|_{\P^1_{\bu}-j(Y)_{\bu}}$ is trivial.  Finally, by assumption of the Theorem we have that
$\cG|_{s(U)}=\cG|_{\infty\times U}$ is trivial.  Thus, all conditions of Lemma~\ref{lem:isotr} are satisfied,
and we conclude that $\cG|_{\P^1_U-j(Y)}$ is trivial.
The theorem is proved.

\end{proof}

\section{Proof of Theorem~\ref{th:inf} and its corollaries}\label{semi-simple}

The aim of this Section is to prove the mentioned theorem and to deduce some corollaries.

For the time being, assume that $U$ is a connected semi-local  affine scheme,
the same as in Proposition \ref{SchemeY22}. Fix a closed subscheme $Z\subset\P^1_U$ finite over $U$.

Let $\bG$ be a simply connected semi-simple $U$-group scheme.
By~\cite[Exp. XXIV, Proposition 5.10]{SGA3} $\bG$ is isomorphic to a direct product
$\prod_{i\in I}\bG_i$ of semi-simple simply connected
$U$-group schemes of the form
$R_{U_i/U}(\bG_i)$,
where each $U_i$ is finite, \'{e}tale over $U$ and connected,
and each $\bG_i$ is simply connected simple $U_i$-group scheme.
We are going to apply Proposition \ref{SchemeY22} to each of the group schemes $\bG_i$ in order
to state and prove Proposition \ref{SchemeY22_new} below.

For each $i\in I$ let $\bu'_i\subset\bu$ be the subset consisting of all closed points $u\in \bu$ such that the simply connected $k(u)$-group
$\bG_{i,u}$ is isotropic.
Let $I'\subset I$ be a subset consisting of all $i\in I$ with non-empty set $\bu'_i$.

First suppose $I'\neq \emptyset$.
For each $i\in I'$
let $Y'_i$ be a finite and \'etale $U$-scheme as in the item (i) of Proposition \ref{SchemeY22}
(applied to the $U$-group scheme $\bG_i$).
In particular, the $Y'_i$-group scheme $\bG_{i,Y'_i}$ is isotropic.

Take a total ordering  on the set $I'$. Applying Proposition \ref{SchemeY22} (iii) to
the schemes $Y'_1$ and $Z$ find
a diagram of the form
$Y'_1\xleftarrow{\pi_1} Y_1\xrightarrow{j_1} \mathbb A^1_U$
such that for each $u\in \bu'_1$ one has $1=\text{g.c.d.}_{v\in Y_{1,u}}[k(v):k(u)]$
and $j_1(Y_1)\sqcup Z=\emptyset$.

Applying Proposition \ref{SchemeY22} (iii) to
the schemes $Y'_2$ and $j_1(Y_1)\sqcup Z$ find
a diagram of the form
$Y'_2\xleftarrow{\pi_2} Y_2\xrightarrow{j_2} \mathbb A^1_U$
such that for each $u\in \bu'_2$ one has $1=\text{g.c.d.}_{v\in Y_{2,u}}[k(v):k(u)]$
and $j_2(Y_2)\cap [j_1(Y_1)\sqcup Z]=\emptyset$.

Continuing this procedure
one can find diagrams
$Y'_i\xleftarrow{\pi_i} Y_i\xrightarrow{j_i} \mathbb A^1_U$
($i\geq 3$)
such that $1=\text{g.c.d.}_{v\in Y_{i,u}}[k(v):k(u)]$ for each $u\in \bu'_i$.
Moreover for
$Y=\sqcup_{i\in I'} Y_i$
the morphism
$$j=\sqcup_{i\in I'}j_i: Y\to \mathbb A^1_U$$
is a closed embedding and $j(Y)\cap Z=\emptyset$.
Put $Y'=\sqcup_{i\in I'}Y'_i$ and
$$
\pi=\sqcup_{i\in I'}\pi_i: Y\to Y'.
$$
If $I'=\emptyset$, then put $Y'=\emptyset=Y$.
\begin{prop}
\label{SchemeY22_new}
Let $U$, $Z\subset \P^1_U$ and \ $\bG=\prod_{i\in I}\bG_i$ be as above in this section.
Then for the diagram $Y'\xleftarrow{\pi} Y\xrightarrow{j} \mathbb A^1_U$ constructed just above the following is true \\
(i) the morphism $\pi$ is finite \'{e}tale and the morphism $j$ is a closed embedding; \\
(ii) for each $i\in I$  with non-empty set $\bu'_i$ the $Y_i$-group scheme $\bG_{i,Y_i}$ is isotropic;\\
(iii) if $\bu'_i\neq \emptyset$, then $1=\text{g.c.d.}_{v\in Y_{i,u}}[k(v):k(u)]$ for each $u\in \bu'_i$; \\
(iv) $j(Y)\cap Z=\emptyset$.

We will often write $Y(\bG)$ for the $U$-scheme $Y$ to indicate that $Y$ depends on $\bG$.
(Following the same principle we have to write $Y'(\bG)$ for the $U$-scheme $Y'$,
$\pi_{\bG}$ for $\pi$ and $j_{\bG}$ for $j$. However, we will not do that.)
\end{prop}

\begin{proof}
The items (i),(iii),(iv) are clear by the very construction of the diagram $(\pi,j)$.
The item (ii) is true since for each $i\in I$  with non-empty the set $\bu'_i$ already
the $Y'_i$-group scheme $\bG_{i,Y'_i}$ is isotropic.
\end{proof}

\begin{thm}\label{th:inf_ssc}
Let $U$, $Z$ and $\bG=\prod_{i\in I}\bG_i$ be as just above.
Let $Y'\xleftarrow{\pi} Y(\bG)\xrightarrow{j} \mathbb A^1_U$ be the diagram
as in Proposition \ref{SchemeY22_new}.
Then for any
principal $\bG$-bundle $\cG$ over $\P^1_U$ with $\cG|_{\infty\times U}$ trivial, the bundle
$\cG|_{\P^1_U-j(Y(\bG))}$ is also trivial and $j(Y(\bG))\cap Z=\emptyset$.
\end{thm}

\begin{rem}
Let $p_i: \bG\to \bG_i$ be the projection and $\cG_i$ be the $\bG_i$-bundle
$(p_i)_*(\cG)$.
It is sufficient to check that for each $i\in I$ the $\bG_i$-bundle
$\cG_i|_{\P^1_U-j(Y(\bG))}$ is trivial.
\end{rem}

\begin{proof}[Proof of Theorem \ref{th:inf_ssc}]
If $i\in I-I'$
(that is $\bu'_i=\emptyset$),
then by the first item of Theorem \ref{th:sc} the $\bG_i$-bundle $\cG_i$
is trivial.
Thus, the $\bG_i$-bundle $\cG_i|_{\P^1_U-j(Y(\bG))}$ is trivial.
It remains to check the case of $i\in I'$.
So, we have to check that the $\bG_i$-bundle
$\cG_i|_{\P^1_U-j(Y(\bG))}$ is trivial in this case.
Recall that $Y(\bG)=\sqcup_{i\in I'} Y_i$ and $j=\sqcup_{i\in I'}j_i: Y(\bG)\to \mathbb A^1_U$.
It suffices to check that the $\bG_i$-bundle
$\cG_i|_{\P^1_U-j(Y_i)}=\cG_i|_{\P^1_U-j_i(Y_i)}$ is trivial. The latter is the case by
Theorem \ref{th:sc}.
The equality
$j(Y(\bG))\cap Z=\emptyset$
holds by Proposition \ref{SchemeY22_new}(iv).
\end{proof}

\begin{proof}[Proof of Theorem \ref{th:inf}]
The claim of Theorem \ref{th:inf} can be proved independently for every connected component of $U$, so we assume that
$U$ is connected.
By Theorem~\ref{th:inf_ssc}, for any closed subscheme $Z\subset\P^1_U$ finite over $U$, there is
a closed subscheme $Y=j(Y(\bG))$ of $\P^1_U$, which is finite and \'etale over $U$, contained in
$\mathbb A^1_U$, satisfies $Z\cap Y=\emptyset$, and such that
all principal $\bG$-bundles over $\P^1_U$ having trivial restriction to $\infty\times U$,
have trivial restriction to $\P^1_U-Y$. Applying the same Theorem~\ref{th:inf_ssc} to the case $Z=Y$,
we find another closed subscheme $Y'$ of $\P^1_U$, such that $Y\cap Y'=\emptyset$ and
all principal $\bG$-bundles over $\P^1_U$ having trivial restriction to $\infty\times U$,
also have trivial restriction to $\P^1_U-Y'$.
It follows that all such $\bG$-bundles also have trivial restriction to $Y$, since it is contaned in $\P^1_U-Y'$. This settles the item $(i)$
of Theorem~\ref{th:inf}. It remains to note that
the isomorphism
classes of $\bG$-bundles over $\P^1_U$ having trivial restriction to $Y$ and $\P^1_U-Y$ are in bijective
correspondence with the set of double cosets
$$
\bG(\P^1_U-Y)\setminus \bG(\dot{Y}^h)\ /\ \bG(Y^h),
$$
since all such bundles are obtained by gluing two trivial $\bG$-bundles over $\P^1_U-Y$ and $Y^h$
by means of an isomorphism taken from $\bG(\dot{Y}^h)$; see~\cite[\S 5]{FP} for the details.
\end{proof}

\begin{proof}[Proof of Theorem \ref{cor:inf_1}]
The existence of $Y'$ and $Y''$ follows from Theorem~\ref{th:inf}.
The set of isomorphism classes of principal $\bG$-bundles
$\cG$ over $\P^1_U$ trivial at infinity equals to the double coset
$$
\bG(\P^1_U-Y')\setminus \bG(\Pro^1_U-(Y'\cup Y''))\ /\ \bG(\P^1_U-Y''),
$$
since $\P^1_U-Y'$ and $\P^1_U-Y''$ constitute a Zariski cover of $\P^1_U$.
\end{proof}

The following theorem is an easy corollary of the above statements.

\begin{thm}\label{th:main}
Let $U$, $\bG$ be as in Theorem~\ref{th:inf}
and $E$ be principal $\bG$-bundle over $U$.
Let $Z\subset \P^1_U$ be a a closed subscheme finite over $U$.
Then there is a closed subscheme $Y\subset\mathbb A^1_U$
which is finite and \'etale over $U$, such that
$Y\cap Z=\emptyset$ and
for any
$\bG$-bundle $\cG$ over $\P^1_U$ with $\cG|_{\infty\times U}\cong E$
the $\bG$-bundle
$\cG|_{\P^1_U-Y}$ is extended from $U$.

Moreover, $\cG|_{\P^1_U-Y}$ is the extension of the principal $\bG$-bundle $E$.
\end{thm}

\begin{notation}\label{G(E)}
Let $U$, $\bG$ be as in Theorem~\ref{th:inf}.
Let $E$ be principal $\bG$-bundle over $U$.
Write $\bG(E)$ for the inner twisted form of the $U$-group scheme $\bG$ corresponding to the principal
$\bG$-bundle $E$. For a principal $\bG$-bundle $\cG$ over $\P^1_U$
we will write $\cG|_{\infty\times U}\cong E$ if
$\cG|_{\infty\times U}$ and $E$ are isomorphic as the principal $\bG$-bundles over $U$.
\end{notation}

\begin{proof}[Proof of Theorem \ref{th:main}]
The claim of Theorem~\ref{th:main} can be proved independently for every connected component of $U$, so we assume that
$U$ is connected. From now on, we use the same notation as in Proposition \ref{SchemeY22_new},
and Notation~\ref{G(E)}.

{\it Claim}. Let $\cG$ be a principal $\bG$-bundle over $\P^1_U$ with $\cG|_{\infty\times U}\cong E$.
Then the $\bG$-bundle
$\cG|_{\P^1_U-j(Y(\bG(E)))}$ is extended from $U$ and
it is the extension of the $\bG$-bundle $E$.
Moreover, $j(Y(\bG))\cap Z=\emptyset$.\\
Clearly, this Claim yields Theorem \ref{th:main}.
It remains to prove the Claim. We will do this in the remaining part of the proof.
Consider the following set bijections
$$H^1_{et}(\P^1_U,\bG) \xrightarrow{Twist^E_{\P^1\times U}} H^1_{et}(\P^1_U,\bG(E)),$$
$$H^1_{et}(U,\bG) \xrightarrow{Twist^E_U} H^1_{et}(U,\bG(E)).$$
Put
$\cG(E)=Twist^E_{\P^1\times U}(\cG)$.
Clearly,
$\cG(E)|_{\infty\times U}$
is a trivial $\bG(E)$-bundle.
By Theorem \ref{th:inf_ssc}
the $\bG(E)$-bundle
$\cG(E)|_{\P^1_U-j(Y(\bG(E)))}$ is trivial.
In particular, the $\bG(E)$-bundle \\
$\cG(E)|_{\P^1_U-j(Y(\bG(E)))}$ is extended from $U$.
Thus, the the $\bG$-bundle
$\cG|_{\P^1_U-j(Y(\bG(E)))}$ is also extended from $U$.
Since $\infty\times U\subset \P^1_U-j(Y(\bG(E)))$ it follows that
$\cG|_{\P^1_U-j(Y(\bG(E)))}$ is the extension of the $\bG$-bundle $E$.
By Proposition \ref{SchemeY22_new}(iv) one has
$Z\subset \P^1_U-j(Y(\bG(E)))$,
the Claim is proved.
The theorem is proved.

\end{proof}

Theorem~\ref{th:main} has the following immediate corollaries which will be used later.

\begin{cor}
Let $U$, $\bG$, $E$ be as in Theorem~\ref{th:main}.
Then there are two closed subschemes $Y'$ and $Y''$ in $\mathbb A^1_U$,  finite and \'etale over $U$, such that $Y'\cap Y''=\emptyset$ and
for any
principal $\bG$-bundle $\cG$ over $\P^1_U$ with $\cG|_{\infty\times U}\cong E$ the bundles
$\cG|_{\P^1_U-Y'}$, $\cG|_{\P^1_U-Y''}$ are extended from $U$.

Moreover,
they both are extensions of the $\bG$-bundle $E$.
\end{cor}

\begin{cor}\label{two_sections}
Let $U$ and $\bG$ be as in Theorem~\ref{th:main}. For any
principal $\bG$-bundle $\cG$ over $\P^1_U$, for any two sections $s_1,s_2:U\to\P^1_U$ of the projection
$\P^1_U\to U$, the pull-backs $s_1^*(\cG)$ and $s_2^*(\cG)$ are isomorphic as principal $\bG$-bundles over $U$.
\end{cor}

\section{On principal $\bT$-bundles over $\P^1_U$}\label{sect:t_sect}

The aim of this section is to prove the following theorem on toral bundles over $\P^1_U$.

\begin{thm}\label{T_bundle_main}
Let $U$ be a semi-local  scheme and $Z\subset\P^1_U$ be a closed subscheme finite over $U$.
Then the following is true \\
(i) there are closed subschemes $S$ and $S'$ in $\mathbb A^1_U$ finite and \'{e}tale over $U$ such that \\
\indent (a) $S\cap S'=\emptyset$; \\
\indent (b) the degrees $[S:U]$ and $[S':U]$ are co-prime integers; \\
\indent (c) $Z\cap(S\sqcup S')=\emptyset$. \\
(ii) for each torus $\bT$ over $U$ and each principal $\bT$-bundle $\cT$ over $\P^1_U$
and each pair $S$, $S'$ as in the item (i)
the $\bT$-bundle $\cT|_{\P^1_U-(S\sqcup S')}$ is extended from $U$.
\end{thm}

We proceed to prove Theorem~\ref{T_bundle_main}. The item (i) of the theorem follows from
the following lemma.

\begin{lem}\label{lem:T_exist}
Let $Z\subset \P^1_U$ be a closed subscheme finite over $U$.
Then there exists a closed subscheme $S'$ in $\mathbb A^1_U$ finite and \'{e}tale over $U$ and such that
$Z\cap S'=\emptyset$. Furthermore, for each such $S'$
there exists a closed subscheme $S\subset \mathbb A^1_U$, finite and \'{e}tale over $U$ such that \\
(i) $S\cap Z=\emptyset=S\cap S'$; \\
(ii) the degrees $[S:U]$ and $[S':U]$ are co-prime integers.
\end{lem}
\begin{proof}[Proof of Lemma \ref{lem:T_exist}]
We prove the existence of $S$ assuming $Z$ and $S'$ are given. The same procedure applied with $S'=\emptyset$
(and no restrictions on the degree) can be used to produce the initial scheme $S'$.

Let ${\bf u}\subset U$ be the set of all closed points in $U$.
For a point $u\in {\bf u}$ let $k(u)$ be its residue field.
Let $\bu_{fin}\subset \bu$ consists of all $u\in \bu$ such that the field $k(u)$ is finite.
Put $\bu_{inf}=\bu-\bu_{fin}$.
Let $u\in \bu_{fin}$. Then the field
$k(u)$ is finite.
For a positive integer $r$ let $k(u)(r)$
be a unique field extension of the degree $r$ of the field $k(u)$.

Let $q\gg [S':U]$ be a prime number such that
for each $u\in \bu_{fin}$ the number $q$ is strictly greater than any of the degrees
$[k(u)(z): k(u)]$ and $[k(u)(s'): k(u)]$, where $z$ runs over all closed points of $Z_u$
and $s'$ runs over all closed points of $S'_u$.

There is a finite \'{e}tale $\bu$-scheme $\bs$ of degree $q$ over $\bu$
such that \\
(1) for $u\in \bu_{fin}$ one has $\bs_u=\spec k(u)(q)$;\\
(2) for $u\in \bu_{inf}$ one has $\bs_u=\sqcup^{q}_{1}u$.

Clearly, there is a finite \'{e}tale $U$-subscheme $S\subset \mathbb A^1_U$ of the form $\spec \cO[T]/(F(T))$
such that the polynomial $F(T)$ is monic of degree $q$ and such that \\
$(*_S)$ for each point $u\in \bu$ one has $S_u=\bs_u$.

Moreover, one can find
$F(T)$ of degree $q$ such that for each
$u\in \bu_{inf}$ one has
$S_u\cap S'_u=\emptyset=S_u\cap Z_u$.
For each $u\in \bu_{fin}$ one has
$S_u\cap S'_u=\emptyset=S_u\cap Z_u$ automatically.
Indeed, by the properties $(*_S)$ and (1) for each
$u\in \bu_{fin}$ the fibre $S_u$ equals $\spec k(u)(q)$.
The latter is a one point set in $\mathbb A^1_{k(u)}$ of the degree $q\gg 0$ choosen above.

The degrees $[S:U]=q$ and $[S':U]$ are co-prime integers by the choice of $q$.
We checked that the $S\subset \mathbb A^1_U$ constructed in this proof
does satisfy the properties (i) and (ii) of the Lemma.
\end{proof}

The proof of Theorem~\ref{T_bundle_main} also uses the following statement which seems to be folklore.
Nevertheless, we decided to include its proof to the sake of completeness.

\begin{prop}\label{Ker_P1 to Ker A1}
Let $\bT$ be a torus over $U$ and $S$ be a $U$-scheme.
Let $\cT$ be a $\bT$-bundle over $\P^1_S$ such that
$\cT|_{\infty\times S}$ is trivial.
Then
$\cT|_{\Aff^1_S}$ is trivial.
\end{prop}
\begin{proof}
To prove this Proposition
consider two presheaves $Ker^{\Pro}$ and $Ker^{\Aff}$ on $Sch/U$:\\
$$V\mapsto Ker[H^1_{et}(\Pro^1_V,\bT)\xra{i^*_0} H^1_{et}(V,\bT)]\ \ \ \text{and}\ \ \  V\mapsto Ker[H^1_{et}(\Aff^1_V,\bT)\xra{i^*_0} H^1_{et}(V,\bT)].$$
It suffices to check that the natural presheaf morphism
$Ker^{\Pro}$ and $Ker^{\Aff}$
{\it vanishes}. Let $V\in Sch/U$.
The Leray spectral sequence yields the following
exact sequence
\begin{equation}\label{eq:T3}
0\to H^1(V,p_*p^*\bT)\xrightarrow{p^*} H^1(\P^1_V,p^*\bT)\xrightarrow{\beta} H^0(V,R^1p_*(p^*\bT))
\end{equation}
showing that the morphism $\beta|_{Ker^{\Pro}(V)}: Ker^{\Pro}(V)\to H^0(V,R^1p_*(p^*\bT))$ is injective. Thus,
the presheaf $Ker^{\Pro}$ is a separated presheaf (for the \'{e}tale topology on $Sch/U$). By the same reasoning
the presheaf $Ker^{\Aff}$ is a separated presheaf (for the \'{e}tale topology on $Sch/U$).
Thus, to check the vanishing of the morphism
$res: Ker^{\Pro} \to Ker^{\Aff}$
it suffices to check it on the \'{e}tale stalks of the presheaves
$Ker^{\Pro}$ and $Ker^{\Aff}$. So, we may suppose that $V$ is local and strictly henselian. In this case
$T=\mathbb G^k_m$ ($k$ is the rank of $\bT$),
$Ker^{\Pro}=Pic(\Pro^1_V)^k$ and the map
$Ker^{\Pro}(V) \xra{res} Ker^{\Aff}(V)$ vanishes because the map
$Pic(\Pro^1_V)\xra{res} Pic(\Aff^1_V)$ vanishes.
\end{proof}

\begin{proof}[Proof of Theorem~\ref{T_bundle_main}]
The item (i) follows from Lemma~\ref{lem:T_exist}.

We proceed to prove (ii).
Let $\bT$ be a torus over $U$. Let $q_U:\P^1_U\to U$ denote the canonical projection.
Consider the restriction $\cT_\infty=\cT|_{\infty\times U}$ which is a
$\bT$-bundle over $U$. Then $\cT'=\cT\cdot q_U^*(\cT_\infty)^{-1}$ is a $\bT$-bundle over $\P^1_U$ such that
$\cT'|_{\infty\times U}$ is trivial.

Observe that the projection $q_S:\P^1_S\to S$ has a section
$$
S\to\Delta(S)\subset S\times_U S\subset\P^1_S.
$$
It follows that $\P^1_S-\Delta(S)$ is isomorphic to $\mathbb A^1_S$. Note that, since $S$ is finite and \'etale over $U$,
$S$ is a semi-local  affine scheme.
The the $\bT$-bundle
$\cT'_S|_{\infty\times S}$ is trivial. Thus, by Proposition
\ref{Ker_P1 to Ker A1}
the $\bT$-bundle $\cT'_S|_{\P^1_{S}-\Delta(S)}$ is trivial.
Consider the trace map on \'etale cohomology
$$
Tr_{S/U}:H^1((\P^1_U-(S\sqcup S'))_S,\bT_S)\to H^1(\P^1_U-(S\sqcup S'),\bT).
$$
Then $Tr_{S/U}\bigl(\cT'_S|_{(\P^1_U-(S\sqcup S'))_S}\bigr)=[S:U]\cdot [\cT']|_{\P^1_U-(S\sqcup S')}$ is trivial.

In the same way $Tr_{S'/U}\bigl([\cT'_{S'}]|_{(\P^1_U-(S\sqcup S'))_{S'}}\bigr)=
[S':U]\cdot [\cT']|_{\P^1_U-(S\sqcup S')}$ is also trivial. Since $[S:U]$ and $[S':U]$ are co-prime integers, it follows that
$[\cT']|_{\P^1_U-(S\sqcup S')}$ is trivial.

Since $[\cT]=[\cT']\cdot q_U^*([\cT_\infty])$, we conclude that
$[\cT]|_{\P^1_U-(S\sqcup S')}$ is extended from $U$.
The proof of Theorem~\ref{T_bundle_main} is completed.
\end{proof}

In the proof of Lemma \ref{H2_C} below we use the following extension of
Proposition~\ref{Ker_P1 to Ker A1}.

\begin{prop}\label{prop:formula}
Let $V$ be a  affine scheme and Let $p:\P^1_V\to V$ denote the canonical projection.
Let $\bT$ be a torus over $V$. Then there is a functorial in $V$ isomorphism
$$p^*+\rho^T_V: H^1(V,\bT)\oplus Hom_{V-gr}(\mathbb G_{m,V},\bT)\to H^1(\P^1_V,\bT),$$
$$\cT \mapsto p^*(\cT) \ \text{and} \ \phi \mapsto \phi_*(\mathcal{O}_{\P^1_V}(-1)).$$
of abelian groups. In particular, $res^{\Pro^1}_{\Aff^1}(H^1(\P^1_V,\bT))\subset \text{Im}[H^1(V,\bT)\to H^1(\Aff^1_V,\bT)]$.
\end{prop}

\begin{proof}
Let $p:\P^1_V\to V$ denote the canonical projection. The Leray spectral sequence of \'etale cohomology
$
H^i(V,R^jp_*p^*\bT)\implies H^{i+j}(\P^1_V,p^*\bT).
$
gives rise to a short exact sequence
\begin{equation}\label{eq:T5}
0\to H^1(V,p_*p^*\bT)\xrightarrow{p^*} H^1(\P^1_V,p^*\bT)\xrightarrow{\beta^T_V} H^0(V,R^1p^*\bT))
\end{equation}
Suppose for a moment that the map $\beta^T_V \circ \rho^T_V$ is an isomorphism.
Then the map $\beta^T_V$ is surjective. Thus, the map
$(p^*,\rho^T_V): H^1(\P^1_V,p^*\bT)\to H^1(V,p_*p^*\bT)\oplus H^0(V,R^1p^*\bT))$
is an isomorphism. The composite map
$$(i^*_0,\beta^T_V)\circ (p^*+\rho^T_V): H^1(V,\bT)\oplus Hom_{V-gr}(\mathbb G_{m,V},\bT)\to H^1(V,\bT)\oplus H^0(V,R^1p^*\bT))$$
is also an isomorphism, because $0=\beta^T_V \circ p^*$. The conclusion is: the map $(p^*+\rho^T_V)$
is an isomorphism.
So, it remains to show that the map $\beta^T_V \circ \rho^T_V$ is an isomorphism.

The asignment
$V\mapsto can^T_V=\beta_V \circ \rho_V$ is an \'{e}tale sheaf morphism
$can^T: \underline {Hom}(\mathbb G_{m},T)\to R^1p_*(T)$.
Thus, it's sufficient to prove that this \'{e}tale sheaf morphism
$$can^T: \underline {Hom}(\mathbb G_{m},T)\to R^1p_*(T)) \ \ [ \ can^T_X(\phi)=\phi_*(\cO(-1)) \ ]$$
is an isomorphism on stalks. So, it's suffices to take a local strictly
henselian $U$-scheme $V$ and to check that the $can^T_V$ is a group isomorphism. Note that for a given
$U$-scheme $X$ the homomorphism $can^T_X$ is functorial in the $X$-torus $T_X=T\times_U X$. Furthermore,
$can^T_X=can^{T'}_X \times can^{T''}_X$, if $T_X=T'\times_X T''$. To complete the proof it remains to note that
the homomorphism
$$can^{\mathbb G_m}_V: \mathbb Z=\text{Hom}(\mathbb G_{m,V},\mathbb G_{m,V})\to \text{H}^1(\Pro^r_V,\mathbb G_m)=\mathbb Z$$
is the identity map.
\end{proof}

\section{Reductive group scheme case}\label{sec:red}
The goal of this section is to prove
Theorem~\ref{th-two-pullbacks} and its corollary
Theorem~\ref{th:zar-cover-red}.

Let $\cO$ be a semi-local  ring.
Put $U:=\spec \cO$.
Let $\bG$ be a reductive $U$-group scheme.

\begin{thm}\label{th-two-pullbacks_partial}
Theorem \ref{th-two-pullbacks} is true if $\bG$ is a $U$-torus or a simply connected semi-simple $U$-group scheme.
\end{thm}

\begin{proof}[Proof of Theorems \ref{th-two-pullbacks} and \ref{th-two-pullbacks_partial}]
Proposition \ref{TwoThmsEquiv}
below shows that
Theorems
\ref{th-two-pullbacks}
and
\ref{th-two-pullbacks_partial}
are equivalent to each other.
Theorem \ref{th-two-pullbacks_partial} for simply connected semisimple groups is true by Corollary \ref{two_sections}
and Lemma \ref{4_properties} below.
Theorem \ref{th-two-pullbacks_partial} for tori is true by
Theorem \ref{T_bundle_main} and Lemma \ref{4_properties}.
Thus, Theorem \ref{th-two-pullbacks} is true as well.
\end{proof}

\begin{prop}
\label{TwoThmsEquiv}
Theorems
\ref{th-two-pullbacks}
and
\ref{th-two-pullbacks_partial}
are equivalent to each other.
\end{prop}

\begin{proof}
Clearly, Theorem \ref{th-two-pullbacks} yields Theorem \ref{th-two-pullbacks_partial}.
In order to derive Theorem \ref{th-two-pullbacks} from Theorem \ref{th-two-pullbacks_partial}, we need the
following obvious
\begin{lem}\label{4_properties}
Let $\bH$ be an affine $U$-group scheme. Then the following properties are equivalent: \\
(i) for any principal $\bH$-bundle $\cH$ over $\P^1_U$, any semi-local  $U$-scheme $V$, and any two $U$-morphisms
$\phi_1,\phi_2: V\rightrightarrows \mathbb P^1_U$, the two principal $\bH$-bundles
$\phi^*_1(\cH)$ and $\phi^*_2(\cH)$ over $V$ are isomorphic; \\
(ii) for any principal $\bH$-bundle $\cH$ over $\P^1_U$, any semi-local  $U$-scheme $V$, and
any two sections $s_1,s_2: V\rightrightarrows \mathbb P^1_V$ of the projection $\mathbb P^1_V\to V$,
the two $\bH$-bundles $s^*_1(\cH_V)$ and $s^*_2(\cH_V)$ over $V$ are isomorphic
(here $\cH_V=(id\times f)^*(\cH)$ and $f: V\to U$ is the structure morphism); \\
(iii) for any principal $\bH$-bundle $\cH$ over $\P^1_U$, any semi-local  $U$-scheme $V$, and
any two sections $s_1,s_2: V\rightrightarrows \mathbb P^1_V$ as
in the item (ii) such that $s_1(V)\cap s_2(V)=\emptyset$,
the two $\bH$-bundles $s^*_1(\cH_V)$ and $s^*_2(\cH_V)$ over $V$ are isomorphic;\\
(iv) for any principal $\bH$-bundle $\cH$ over $\P^1_U$,
for any semi-local  $U$-scheme $V\xra{f} U$, and the bundle $\cH_V$ as in (ii),
 the principal $\bH$-bundles
$s^*_0(\cH_V)$ and $s^*_{\infty}(\cH_V)$ over $V$ are isomorphic
(here $s_0$ and $s_{\infty}$ are the zero and the infinity sections respectively);\\
(v) for any principal $\bH$-bundle $\cH$ over $\P^1_U$,
the principal $\bH$-bundles
$s^*_0(\cH)$ and $s^*_{\infty}(\cH)$ over $U$ are isomorphic.
\end{lem}
\begin{notation}\label{n: P_P}
Let $0\neq n\in \mathbb N$ be an integer. Write
${\bf n}: \P^1_W \to \P^1_W$
for the morphism taking
$[t_0:t_1]$ to $[t^n_0:t^n_1]$.
Write
$\P^{1,dom({\bf n})}_W$ for the domain of the morphism ${\bf n}$.\\
Write
$s^{dom({\bf n})}_{\infty}$ and $s^{dom({\bf n})}_{0}$
for the infinity and the zero sections of the projection
$\P^{1,dom({\bf n})}_U \to U$.
Clearly,
${\bf n}\circ s^{dom({\bf n})}_{\infty}=s_{\infty}$,
${\bf n}\circ s^{dom({\bf n})}_{0}=s_{0}$.
\end{notation}

\begin{notation}\label{ext_seq}

Let $\bS$ be the radical $Rad(\bG)$ of $\bG$, let $\bG^{sc}$ be the simply connected cover of the derived group $\bG^{der}$.
Put $\tilde \bG= \bS\times \bG^{sc}$. Let $\Pi: \tilde \bG \to \bG$ be the canonical $U$-group scheme morphism.
It is a central isogeny, that is, if $\bC=Ker(\Pi)$, then $\bC$ is a finite $U$-group scheme of multiplicative type,
central in $\tilde\bG$, and
one has a short exact sequence of $fppf$-sheaves
$\{1\}\to \bC \xra{in} \tilde \bG\xra{\Pi} \bG\to \{1\}$.
\end{notation}

To complete the proof of Proposition \ref{TwoThmsEquiv}
we need also the following lemma
which will be proven below in this text.

\begin{lem}\label{mathcal G_sc}
There is
a positive integer $m$
(depending only of the group $\bG$)
such that
for each
principal $\bG$-bundle $\cG$ over $\P^1_U$ with $\cG|_{\infty\times U}$ trivial
there is
a $\tilde \bG$-bundle $\tilde \cG$ over $\P^{1,dom({\bf m})}_U$ such that \\
(1) ${\bf m}^*(\cG)=\Pi_*(\tilde \cG)$ over $\P^{1,dom({\bf m})}_U$, \\
(2) $(s^{dom({\bf m})}_{\infty})^*(\tilde \cG)$ is trivial.
\end{lem}
Given the last two Lemmas, we are able to complete the proof of Proposition~\ref{TwoThmsEquiv}.
Namely, we are given a reductive $U$-group scheme $\bG$ and a principal
$\bG$-bundle $\cG$ over $\P^1_U$.
We are also given a semi-local $U$-scheme $V$ and two $U$-morphisms
$\phi_1,\phi_2: V\rightrightarrows \mathbb P^1_U$.
We need to check that the $\bG$-bundles
$\phi^*_1(\cG)$ and $\phi^*_2(\cG)$
over $V$ are isomorphic. By Lemma \ref{4_properties}
it suffices to check that $s^*_0(\cG)$ and $s^*_{\infty}(\cG)$
are isomorphic as $\bG$-bundles over $U$,
where $s_0$ and $s_{\infty}$ are the zero and the infinity sections of the canonical projection $p:\P^1_U\to U$
respectively. Twisting $\bG$ by means of
the $\bG$-bundle $s^*_{\infty}(\cG_{new})$ and $\cG$ by means of $p^*s^*_{\infty}(\cG_{new})$,
we obtain a $U$-group $\bG_{new}$ and
a $\bG_{new}$-bundle $\cG_{new}$ such that the bundle $s^*_{\infty}(\cG_{new})$ is trivial.
The application of the inverse twist functor shows that it remains to check that the
$\bG_{new}$-bundle $s^*_0(\cG_{new})$ is trivial.

By Lemma \ref{mathcal G_sc} there is an integer $m>0$ and
a $\tilde \bG_{new}$-bundle $\tilde \cG_{new}$ over $\P^{1,dom({\bf m})}_U$ such that
(1) ${\bf m}^*(\cG_{new})=\Pi_*(\tilde \cG_{new})$ over $\P^{1,dom({\bf m})}_U$ and
(2) $(s^{dom({\bf m})}_{\infty})^*(\tilde \cG_{new})$ is trivial.

Clearly, $\tilde \bG_{new}\cong \bS\times \bG^{sc}_{new}$.
By Theorem \ref{th-two-pullbacks_partial}
the $\tilde \bG_{new}$-bundle
$(s^{dom({\bf m})}_{0})^*(\tilde \cG_{new})$
is trivial. Now by Lemma \ref{mathcal G_sc} the $\bG$-bundle
$(s^{dom({\bf m})}_{0})^*({\bf m}^*(\cG_{new}))$ is trivial.
Since ${\bf m}\circ s^{dom({\bf m})}_{0}=s_0$ the $\bG$-bundle
$s^*_0(\cG_{new})$ is trivial.

The derivation of Theorem \ref{th-two-pullbacks} from Theorem
\ref{th-two-pullbacks_partial} is complete. Proposition~\ref{TwoThmsEquiv} is proved.

\end{proof}

Next we prove Lemma \ref{mathcal G_sc}. To do this it suffices to prove
Lemma \ref{H2_C} below. To state that Lemma we need some more notation.

\begin{notation}\label{proj_infty}
Let $\cF$ be a presheaf of abelian groups on the category $Sch/U$ of $U$-schemes.
Write below in the text
$\cF(\P^1)$ for the presheaf ${\cH}om(\P^1,\cF)$ and $\cF(\P^1)_{\infty}$ for its direct summand
$Ker[s^*_{\infty}: {\cH}om(\P^1,\cF)\to \cF]$, where $s_{\infty}$ is the infinity section of the $\P^1$.

Clearly, the assignments $\cF\mapsto \cF(\P^1)$ and $\cF\mapsto \cF(\P^1)_{\infty}$ are functors on the category
$PreAb$ of presheves of abelian groups on $Sch/U$. The inclusion
$\cF(\P^1)_{\infty}\subseteq \cF(\P^1)$ is a functor transformation on the category $PreAb$
(and hence on the category of sheaves).
\end{notation}

\begin{lem}\label{H2_C}
Let $n$ be the order of $\bC$, that is the degree $[\bC: U]$. Put $m=n^2$. Then the pull back map
${\bf m}^*: H^2_{fppf}(\P^1_U,\bC)_{\infty}\to H^2_{fppf}(\P^{1,dom({\bf m})}_U,\bC)_{\infty}$
vanishes.
\end{lem}
\begin{proof}
Let $\tilde \bT$ a maximal $U$-torus in $\tilde \bG$. By~\cite[Exp.XXII, Cor.4.1.7]{SGA3}
one has $\bC\subset \tilde \bT$. Put $\bT=\tilde \bT/\bC$. Then $\bT$ is a maximal torus of $\bG$
and $\{1\}\to \bC \xra{in} \tilde \bT\xra{\Pi} \bT\to \{1\}$ is a short exact sequence of $fppf$-sheaves.
It induces the following short exact sequence of cohomology groups:
$$
H^1_{fppf}(\P^1_U,\bT)_\infty\to H^2_{fppf}(\P^1_U,\bC)_\infty\to H^2_{fppf}(\P^1_U,\tilde\bT)_\infty.
$$
Note that, since $\bT$ and $\tilde\bT$ are smooth, the \'etale cohomology with values in these sheaves coincides with fppf
cohomology.

Recall that $p:\P^1_U\to U$ is the canonical projection.
The spectral sequence of the form
$H^i(U,R^jp_*(\tilde \bT)_{\infty})\Rightarrow H^{i+j}(\P^1_U,\tilde \bT)_{\infty}$
shows that
the map
$$Ker[H^2(\P^1_U,\tilde \bT)_{\infty}\to H^0(U,R^2p_*(\tilde \bT)_{\infty}]\to H^1(U,R^1p_*(\tilde \bT)_{\infty})
$$
is mono.

The sheaf $R^2p_*(\tilde\bT)_\infty$ has no torsion~\cite[pp. 193--194]{Gabber-thesis}.
Thus, there is a canonical map
$H^2(\P^1_U,\tilde \bT)_{\infty,tors}\to H^1(U,R^1p_*(\tilde \bT)_{\infty})$
and it is mono. By Proposition~\ref{prop:formula} the sheaf
$R^1p_*(\tilde \bT)_{\infty}$ is the cocharacter sheaf $\tilde {\underline{\bT}}^{\vee}$,
and one readily sees that the sheaf morphism
${\bf n}^*: \tilde {\underline{\bT}}^{\vee}\to \tilde {\underline{\bT}}^{\vee}$
is the multiplication by $n$. Thus, the map
${\bf n}^*: H^1(U,R^1p_*(\tilde \bT)_{\infty})\to H^1(U,R^1p_*(\tilde \bT)_{\infty})$
is the multiplication by $n$. Hence
$$Im[H^2(\P^1_U,\bC)_{\infty}\xra{{\bf n_{\bT}}^* \circ \ in_*} H^2(\P^{1,dom({\bf n})}_U,\tilde \bT)_{\infty}]\subseteq$$
$$\subseteq Ker[H^2(\P^{1,dom({\bf n})}_U,\bC)_{\infty}\xra{in_*} H^2(\P^{1,dom({\bf n})}_U,\tilde \bT)_{\infty}]=
$$
$$
=Im[H^1(\P^{1,dom({\bf n})}_U,\bT)_{\infty}\xra{\delta} H^2(\P^{1,dom({\bf n})}_U,\bC)_{\infty}].$$

Also by Proposition~\ref{prop:formula} we have that
$H^1(\P^{1,dom({\bf n})}_U,\bT)_{\infty}=\underline{\bT}^{\vee}(U)$.
This yields that the map
${\bf n}^*_{\bT}: H^1(\P^{1,dom({\bf n})}_U,\bT)_{\infty}\to H^1(\P^{1,dom({\bf n})}_U,\bT)_{\infty}$
is the multiplication by $n$.
So, for each $a\in H^2(\P^1_U,\bC)_{\infty}$ there is $\xi\in \underline{\bT}^{\vee}(U)$ with $\delta(\xi)={\bf n}^*_{\bC}(a)$.
Therefore one has
${\bf m}^*_{\bC}(a)={\bf n}^*_{\bC}({\bf n}^*_{\bC}(a))={\bf n}^*_{\bC}(\delta(\xi))=\delta({\bf n}^*_{\bT}(\xi))=\delta(n\xi)=n\delta(\xi)=0$
(the very last equality uses the fact that $H^2(\P^{1,dom({\bf n})}_U,\bC)_{\infty}$ is the exponent $n$ group).
\end{proof}

Theorem \ref{th-two-pullbacks} has the following consequences.

\begin{cor}\label{cor:Z_1_and_Z_2}
Let $U$, $\bG$ and $\cG$ be as in Theorem~\ref{th-two-pullbacks}.
Let $Z\subset \P^1_U$ be a closed subset finite over $U$.
Then there is a Zariski neighborhood $V$ of $Z$ in $\P^1_U$
such that the $\bG$-bundle $\cG|_{V}$ is extended from $U$, and all fibres of $V$ over the closed points of $U$ are
non-empty. If $E=\cG|_{\infty\times U}$, then $\cG|_{V}$ is the extension of the $\bG$-bundle $E$ via
the canonical morphism $V\to U$.
\end{cor}

\begin{cor}
\label{cor:V_1_and_V_2}
Let $U$, $\bG$ and $\cG$ be as in Theorem~\ref{th-two-pullbacks}.
Let $Z$ and $V$ be as in Corollary \ref{cor:Z_1_and_Z_2}.
Put $Z'=\P^1_U-V$. Then $Z'$ is finite over $U$.
Let $V'$ be the Zariski neighborhood of $Z'$ in $\P^1_U$ that exists by Corollary \ref{cor:Z_1_and_Z_2}. Then
$V\cup V'=\P^1_U$ and the $\bG$-bundles $\cG|_{V}$ and $\cG|_{V'}$ are extended from $U$.
Moreover, if $E=\cG|_{\infty\times U}$, then these two bundles are both extensions of the $\bG$-bundle $E$.
\end{cor}

\begin{proof}[Proof of Corollary \ref{cor:Z_1_and_Z_2}]
If $Z$ is empty, then there is nothing to prove. So, let $Z$ be non-empty.
If one of the fibres of $Z$ over closed points of $U$ is empty,
then adjoin to $Z$ a section of the projection $p:\P^1_U\to U$.
Since $Z$ is finite over $U$ it is semi-local.
Let $z_1,...,z_s$ be the set of all closed points of $Z$.
Put ${\cal V}=Spec ({\cal O}_{\P^1_U;z_1,...,z_s})$.
Then $\cal V$ is a semi-local  $U$-scheme.
Consider two $U$-morhisms
$\phi_1: {\cal V}\xra{in} \P^1_U$
and
$\phi_2: {\cal V}\xra{s_{\infty}\circ (p|_{\cal V})} \P^1_U$.
Put $E=\cG|_{\infty\times U}$.
By Theorem \ref{th-two-pullbacks} the two $\bG$-bundles
$\phi^*_1(\cG)$ and $\phi^*_2(\cG)$ over $\cal V$ are isomorphic.
Thus, $\cG|_{V}$ is the extension of the $\bG$-bundle $E$.
Hence there is a Zariski neighborhood $V$ of $Z$ in $\P^1_U$
such that
the $\bG$-bundles $\cG|_{V}$ is the extension of the $\bG$-bundle $E$.

\end{proof}

\begin{proof}[Proof of Corollary \ref{cor:V_1_and_V_2}]
This Corollary is an easy consequence of Corollary \ref{cor:Z_1_and_Z_2}.
\end{proof}

\begin{proof}[Proof of Theorem~\ref{th:zar-cover-red}]
The statement of this theorem is just a reformulation of Corollary~\ref{cor:V_1_and_V_2}.
\end{proof}

\begin{proof}[Proof of Theorem~\ref{th:X}]
Clearly, we may suppose that $X$ is local.
In this case Theorem \ref{th:zar-cover-red} yields the result.
\end{proof}

\section{Appendix}
Let $U$ be as in Proposition \ref{SchemeY22}.
Put $\cO=\Gamma(U,\cO_U)$.
The main aim of this Section is to prove
Proposition \ref{SchemeY222}. By Remark
\ref{222yields22} the latter Proposition yields
the items (ii) and (iii) of Proposition \ref{SchemeY22}.
The results of this section are heavily based on Bertini type theorems
proved in \cite{Poo1} and \cite{Poo2}.

Let ${\bf u}\subset U$ be the set of all closed points in $U$.
For a point $u\in {\bf u}$ let $k(u)$ be its residue field.
Let $\bu_{fin}\subset \bu$ consists of all $u\in \bu$ such that the field $k(u)$ is finite.

\begin{defn}\label{bw_Y'}
Let $\bw\subset \bu$ be a subset and let $Y'$ be a $U$-scheme.
One says that the pair $(\bw, Y')$ is {\bf nice} if $Y'$ is finite, \'etale over $U$, $\bu_{fin}\subset \bw$
and for each $u\in \bw$ the $u$-scheme $Y'_u$ contains a $k(u)$-rational point.
Put $\bw_{inf}=\bw-\bu_{fin}$ and $\bu_{inf}=\bu-\bu_{fin}$.
\end{defn}

\begin{rem}\label{222yields22}
Let $\bG$, $\bu'\subset \bu$ and $Y'$ be as in Proposition \ref{SchemeY22}.
Suppose $Y'$ satisfies the property (i) of that Proposition.
Then the pair $(\bu',Y')$ is nice.
Indeed, for each $u\in \bu_{fin}$ the $k(u)$-group $\bG_u$ is quasi-split. In particular, it is isotropic.
Thus, $\bu_{fin}\subset \bu'$. Finally, for each $u\in \bu'$ the $u$-scheme $Y'_u$ contains a $k(u)$-rational point
by our assumption on $Y'$.

So, Proposition \ref{SchemeY222} yields
the items (ii) and (iii) of Proposition \ref{SchemeY22}.
\end{rem}

\begin{prop}
\label{SchemeY222}
Let $U$ be as in Proposition \ref{SchemeY22}.
Let $\bw\subset \bu$ be a subset and $Y'$ be a $U$-scheme.
Suppose the pair $(\bw, Y')$ is nice. Then one has\\
(i) there are diagrams of the form
$Y'\xleftarrow{\pi} Y\xrightarrow{j} \mathbb A^1_U$
with a finite \'{e}tale morphism $\pi$ and a closed $U$-embedding $j$
such that   \\
(a) for each $u\in \bw$ one has $1=\text{g.c.d.}_{v\in Y_u}[k(v):k(u)]$ and  \\
(b) for each $u\in \bw_{inf}$ the scheme $Y_u$ has a $k(u)$-rational point; \\
(ii) let $Z\subset \P^1_U$ be a closed $U$-finite subscheme. Then there are diagrams $Y'\xleftarrow{\pi} Y\xrightarrow{j} \mathbb A^1_U$
as in the assertion (i) of this proposition
and such that
$j(Y)\cap Z=\emptyset$.
\end{prop}
This Proposition is a direct consequence of Lemma \ref{F1F2} stated and proven below in this Section.
In turn, Lemma \ref{F1F2} is a slight extension of \cite[Lemma 2.1]{Pan3}.

\begin{notation}\label{notn: F1F2}
Let $k$ be the finite field.
For a positive integer $r$ let $k(r)$
be a unique field extension of the degree $r$ of the field $k$.
Let $\mathbb A^1_{k}(r)$ be the set of all degree $r$ points
on the affine line $\mathbb A^1_{k}$.
Let $Irr_k(r)$ be the number of the degree $r$ points on
$\mathbb A^1_{k}$.
\end{notation}

\begin{lem}
\label{F1F2}
Let $U$, $\bw\subset \bu$ and $Y'$ be as in Proposition \ref{SchemeY222}. Suppose the pair $(\bw,Y')$ is nice.
Let $Z\subset\mathbb A^1_U$ be a closed subscheme finite over $U$.
Then
there are finite \'{e}tale $U$-schemes $U^{(r)}$ of the form $\spec \cO[T]/(F_r(T))$ ($r=1,2$)
with monic polynomials $F_r(T)$ of prime degrees $q_r$ ($r=1,2$) such that \\
(i) the degrees $q_1$ and $q_2$ are coprime and for each $u\in \bu_{fin}$ one has $U^{(r)}_u=\spec k(u)(q_r)$; \\
(ii) for any $u\in \bu_{fin}$ and $v\in Y'_u$ one has $k(v)\otimes_{k(u)} k(u)(q_r)$ is a field for $r=1,2$; \\
(iii) for each $u\in \bu_{fin}$ the degrees $q_1$ and $q_2$ are strictly greater than any of the degrees
$[k(u)(z): k(u)]$, where $z$ runs over all closed points of $Z_u$; \\
(iv) for each $u\in \bu_{inf}$ and for $r=1,2$ one has $U^{(r)}_u=\sqcup^{q_r}_{1}u$; \\
(v) there is a closed embedding of $U$-schemes
$$Y=Y^{\prime}\times_U U^{(1)} \coprod Y^{\prime}\times_U U^{(2)} \xrightarrow{j} \mathbb A^1_U,$$\\
such that one has $j(Y) \cap Z = \emptyset$; \\
(vi) for each $u\in \bu_{fin}$ one has
$1=\text{g.c.d.}_{v\in Y_u}[k(v):k(u)]$; \\
(vii) for each $u\in \bw_{inf}$ the $k(u)$-scheme $Y_u$ contains a $k(u)$-rational point;\\
(viii) $Pic(\P^1_{u}-j(Y)_{u})=0$ for any closed $U$-embedding $j: Y\hra \mathbb A^1_U$ and any $u\in \bw$.
\end{lem}
This lemma is proved at the end of this Section.
Its proof requires some preparation.

\begin{lem}\label{l:F1F2_preliminary}
Let $k$ be a finite field. Let $c=\sharp(k)$ (the cardinality of $k$).
Let $d\geq 1$ be an integer and
$k(d)/k$ be a unique finite field extension of degree $d$.
Let $q\in \mathbb N$
be a prime which is co-prime as to the characteristic of the field $k$,
so to the integer $d$.
Then \\
(1) $Irr_k(q)=(c^q-c)/q$;\\
(2) the $k$-algebra $k(d)\otimes_{k} k(q)$ is the field $k(d)(q)=k(dq)$;\\
(3) $Irr_k(dq) \geq Irr_k(q)$;\\
(4) $Irr_k(q)=(c^q-c)/q \gg 0$ for $q\gg 0$.
\end{lem}

\begin{proof}
The assertions (1) and (2) are clear. The assertion (3) is clear for $d=1$.
To prove the assertion (3) for $d\geq 2$ note that
$Irr_k(dq)=(c^{dq}-c^d-c^q+c)/dq$. By (1) we know that $Irr_k(q)=(c^q-c)/q$.
It remains to check that
$c^{dq}-c^d \geq (d+1)(c^q-c)$.
Which in turn is
equivalent
to the inequality
$c^d(c^{d(q-1)}-1)\geq (d+1)(c^q-c)$. The latter is true.
So, the assertion (3) is proved.
It yields the assertion (4).
These proves the lemma.
\end{proof}

\begin{notation}
\label{d(Y_u)}
For any \'{e}tale $k(u)$-scheme $W$
set $d(W)=\text{max} \{ deq_{k(u)} k(u)(v)| v\in W \}$.
\end{notation}

\begin{lem}\label{l:F1F2_surjectivity}
Let $u\in \bu_{fin}$ (so, the residue field $k(u)$ is finite).
Let $Y'_u$ be an \'{e}tale $k(u)$-scheme. For any positive integer $d$ let
$Y'_u(d)\subseteq Y'_u$ be the subset consisting of points $v\in Y'_u$ such that
$deg_{k(u)}(k(v))=d$.
For any prime $q\gg 0$ the following holds: \\
(1) if $v\in Y'_u$, then $k(v)\otimes_{k(u)} k(u)(q)$ is the field $k(v)(q)$ of the degree $q$ over $k(v)$; \\
(2) $k(u)[Y'_u]\otimes_{k(u)} k(u)(q)=
\prod_{v\in Y'_u} k(v)(q)=\prod^{d(Y'_u)}_{d=1}\prod_{v\in Y'_u(d)} k(u)(dq)$;\\
(3) there is a surjective
$k(u)$-algebra homomorphism
\begin{equation}\label{eq:first_surjectity}
\alpha: k(u)[t]\to k(u)[Y'_u]\otimes_{k(u)} k(u)(q)=\prod^{d(Y'_u)}_{d=1}\prod_{v\in Y'_u(d)} k(u)(dq).
\end{equation}
\end{lem}

\begin{proof}
The assertions (1) and (2) follows from the lemma \ref{l:F1F2_preliminary}(2).

We now prove the third assertion. By statements (3) and (4) of Lemma \ref{l:F1F2_preliminary}
there is a prime $q\gg 0$ such that for any $d=1,...,d(Y'_u)$ one has
$Irr_{k(u)}(dq)\geq \sharp(Y'_u(d))$.
In this case for any $d=1,2,...,d(Y'_u)$ there exists
a surjective $k(u)$-algebra homomorphism
$\prod_{x\in \mathbb A^1_{k(u)}(dq)} k(u)(dq)\to \prod_{v\in Y'_u(d)}k(u)(dq)$.
Thus for this specific choice of prime $q$ there exists a surjective $k(u)$-algebra homomorphism
$k(u)[t]\to k(u)[Y'_u]\otimes_{k(u)} k(u)(q)$.
The third assertion of the lemma is proved.
\end{proof}

\begin{lem}
\label{F1F2_copy_very_final}
Let $U$ be as in Proposition \ref{SchemeY222}.
Let $u \in \bu_{fin}$ be a closed point and let $k(u)$ be its residue field.
Let $Y^{\prime}_u \to u$ be a finite \'{e}tale morphism such that
$Y^{\prime}_u$
contains a $k(u)$-rational point. Then for any
two different primes $q_1,q_2\gg 0$
the finite field extensions
$k(u)(q_1)$ and $k(u)(q_2)$ of the finite field $k(u)$ satisfy the following conditions \\
(i) for $r=1,2$ one has $q_r > d(Y'_u)$;\\
(ii) for any $r=1,2$ and any point $v\in Y'_u$ the $k(u)$-algebra $k(v) \otimes_{k(u)} k(u)(q_r)$ is a field;\\
(iii) both primes $q_1,q_2$ are strictly greater than
$\text{max}\{[k(u)(z): k(u)]| z\in Z_u \}$;\\
(iv) there is a closed embedding of $u$-schemes
$$Y_u=[(Y^{\prime}_u\otimes_{k(u)} k(u)(q_1) \coprod (Y^{\prime}_u\otimes_{k(u)} k(u)(q_2)] \xrightarrow{j_u} \mathbb A^1_u;$$\\
(v)
one has $j_u(Y_u) \cap Z_u = \emptyset$;\\
(vi) one has $1=\text{g.c.d.}_{v\in Y_u}[k(v):k(u)]$.
\end{lem}

\begin{proof}
Clearly the conditions (i) and (iii) holds for all big enough primes $q_1,q_2$.
Applying now the lemma \ref{l:F1F2_surjectivity}(2) to the \'{e}tale $k(u)$-scheme $Y'_u$
we see that
the condition (ii)
holds also for all big enough primes $q_1,q_2$.
Take now any two different primes $q_1,q_2$
satisfying the conditions (i) to (iii) and such that $q_1,q_2\gg 0$.
We claim that for those two primes the conditions
(iv) to (vi) are satisfied.

Check the condition (iv).
Lemma \ref{l:F1F2_surjectivity}(3)
shows that there are surjections of the $k(u)$-algebras
$\alpha_r: k(u)[t] \to k(u)[Y'_u]\otimes_{k(u)} k(u)(q_r)$ for $r=1,2$.
The surjections $\alpha_r$'s induce two closed embeddings
$j_{u,r}: Y^{\prime}_u\otimes_{k(u)} k(u)(q_r) \hookrightarrow \mathbb A^1_{k(u)}$
of $k(u)$-schemes.
By the lemma \ref{l:F1F2_surjectivity}(1)
for any point $w\in Y^{\prime}_u\otimes_{k(u)} k(u)(q_r)$
one has
$deg_{k(u)}(w)=d\cdot q_r$
for an appropriate $d=1,2,...,d(Y'_u)$.
Since the primes $q_1$, $q_2$ do not coincide and
$q_r > d(Y'_u)$ for $r=1,2$, hence for any
$d_1,d_2 \in \{1,2,...,d(Y'_u)\}$
one has
$d_1q_1\neq d_2q_2$.
Thus the intersection of
$j_{u,1}(Y^{\prime}_u\otimes_{k(u)} k(u)(q_1))$ and $j_{u,2}(Y^{\prime}_u\otimes_{k(u)} k(u)(q_2))$ is empty.
Put
$$Y_u=Y^{\prime}_u\otimes_{k(u)} k(u)(q_1) \sqcup Y^{\prime}_u\otimes_{k(u)} k(u)(q_2).$$
We have checked that the $k(u)$-morphism
$$j_{u}=j_{u,1}\sqcup j_{u,2}: Y_u \to \mathbb A^1_u$$
is a closed embedding. The condition (iv) is verified.
Since for any point $w\in Y^{\prime}_u\otimes_{k(u)} k(u)(q_r)$
one has
$deg_{k(u)}(w)=d\cdot q_r\ge q_r$,
hence the above choice of $q_1$ and $q_2$
yields
the condition (v).

Since $Y^{\prime}_u$
contains a $k(u)$-rational point,
hence
$Y_u$ has a point of the degree $q_1$ over $k(u)$ and a point of the degree $q_2$ over $k(u)$.
Thus, $1=\text{g.c.d.}_{v\in Y_u}[k(v):k(u)]$.
This proves the condition (vi).
\end{proof}

\begin{proof}[Proof of Lemma \ref{F1F2} and Proposition \ref{SchemeY222}]
Let $q_1\gg 0$ be a prime number. There is a finite \'{e}tale $\bu$-scheme $\bu^{(1)}$ of degree $q_1$ over $\bu$
such that \\
(1) for $u\in \bu_{fin}$ one has $\bu^{(1)}_u=\spec k(u)(q_1)$;\\
(2) for $u\in \bu_{inf}$ one has $\bu^{(1)}_u=\sqcup^{q_1}_{1}u$.\\
Clearly, there is a finite \'{e}tale $U$-scheme $U^{(1)}$ of the form $\spec \cO[T]/(F_1(T))$
such that the polynomial $F_1(T)$ is monic of degree $q_1$ such that \\
$(*_{(1)})$ for each point $u\in \bu$ one has $U^{(1)}_u=\bu^{(1)}_u$. \\\\
Let $q_2\gg 0$ be a prime number with $q_2\neq q_1$. There is a finite \'{e}tale $\bu$-scheme $\bu^{(2)}$ of degree $q_2$ over $\bu$
such that \\
(1) for $u\in \bu_{fin}$ one has $\bu^{(2)}_u=\spec (k(u)(q_2))$;\\
(2) for $u\in \bu_{inf}$ one has $\bu^{(2)}_u=\sqcup^{q_2}_{1}u$.\\
Clearly, there is a finite \'{e}tale $U$-scheme $U^{(2)}$ of the form $\spec \cO[T]/(F_2(T))$
such that the polynomial $F_2(T)$ is monic of degree $q_2$ such that \\
$(*_{(2)})$ for each point $u\in \bu$ one has $U^{(2)}_u=\bu^{(2)}_u$. \\\\
It remains to
check that these $U^{(1)}$ and $U^{(2)}$ subject to the conditions (i)--(viii) of Lemma \ref{F1F2}. This will
prove the Lemma.
Conditions (i),(iii) and (iv) are satisfied by the choice of $U^{(1)}$ and $U^{(2)}$.\\\\
Put $Y=Y'\times_U U^{(1)}\sqcup Y'\times_U U^{(2)}$.
Conditions (ii) and (vi) are satisfied by Lemma \ref{F1F2_copy_very_final}.
Condition (vii) is satisfied since for any $u\in \bw_{inf}$ and for any $r=1,2$
one has $Y'_u(u)\neq \emptyset$ and $\bu^{(r)}_u(u)\neq \emptyset$.  \\\\
Prove now the assertion (v).
If $u \in \bu_{fin}$, then take the closed embedding
$j_u: Y_u \hra \mathbb A^1_u$ as in Lemma \ref{F1F2_copy_very_final}.
If $u \in \bu_{inf}$, then the residue field $k(u)$ is infinite. Thus,
there are many closed embeddings $j_u: Y_u \hra \mathbb A^1_u$.
Particularly, one can take a closed embedding $j_u: Y_u \hra \mathbb A^1_u$
with $j_u(Y_u) \cap Z_u = \emptyset$. Consider the closed embedding
$$j_{\bu}=\sqcup_{u\in \bu}j_u: Y_{\bu}=\sqcup_{u\in \bu}Y_u\hra \sqcup_{u\in \bu}\mathbb A^1_u=\mathbb A^1_{\bu}.$$
Put $k(\bu):=\sqcup_{u\in \bu}k(u)$.
The pull-back map of the $k(\bu)$-algebras
$j^*_{\bu}: k(\bu)[t]\to \Gamma(Y_{\bu},\cO_{Y_{\bu}})$ is surjective.
Consider the element $a=j^*_{\bu}(t)\in \Gamma(Y_{\bu},\cO_{Y_{\bu}})$.
Let $A\in \Gamma(Y,\cO_{Y})$ be a lift of the element $a$.

Consider a unique $\cO$-algebra homomorphism
$\cO[t]\xrightarrow{j^*} \Gamma(Y,\cO_{Y})$
which takes $t$ to $A$.
Since $\Gamma(Y,\cO_{Y})$ is a finitely generated $\cO$-module
and $j^*\otimes_{\cO} k(\bu)=j^*_{\bu}$,
the Nakayama lemma shows that the map $j^*$ is surjective.
Hence the induced scheme morphism
$j: Y\to \mathbb A^1_U$
is a closed embedding.
To complete the proof of the assertion $v$ it remains to check that
$j(Y)\cap Z=\emptyset$.
Since $Y$ and $Z$ are $U$-finite,
it is sufficient to check that
$j(Y)_{\bu}\cap Z_{\bu}=\emptyset$
(here for a $U$-scheme $W$ we write $W_{\bu}$ for the $\bu$-scheme $W\times_U {\bu}$). \\\\
Clearly, the two closed embeddings
$Y_{\bu}\hra Y\xrightarrow{j} \mathbb A^1_U$
and
$Y_{\bu}\xrightarrow{j_{\bu}} \mathbb A^1_{\bu} \hra \mathbb A^1_U$
coincide.
Thus, $j(Y)_{\bu}=j_{\bu}(Y_{\bu})$.
One has $j_{\bu}(Y_{\bu})\cap Z_{\bu}=\emptyset$ by the very choice of
$Y_{\bu}$ and $j_{\bu}$. We checked the equality
$j(Y)\cap Z=\emptyset$.
The assertion (v) is proved.

The assertion (viii) is true since we already verified the ones (vi) and (vii).
Lemma \ref{F1F2} is proved. Thus, Proposition \ref{SchemeY222} is proved too.
\end{proof}

\end{document}